\documentclass{amsart}

\usepackage{thesis} 

\begin{document}

\title{\MakeUppercase{\grtitle}}
\date{\today}
\author{\textsc{Aaron Mazel-Gee}}

\begin{abstract}
We provide, among other things: (i) a Bousfield--Kan formula for colimits in $\infty$-categories (generalizing the 1-categorical formula for a colimit as a coequalizer of maps between coproducts); (ii) $\infty$-categorical generalizations of Barwick--Kan's Theorem {\Bn} and Dwyer--Kan--Smith's Theorem {\Cn} (regarding homotopy pullbacks in the Thomason model structure, which themselves vastly generalize Quillen's Theorem B); and (iii) an articulation of the simultaneous and interwoven functoriality of colimits (or dually, of limits) for natural transformations and for pullback along maps of diagram $\infty$-categories.
\end{abstract}

\maketitle

\papernum{3}

\setcounter{tocdepth}{1}
\tableofcontents

\setcounter{section}{-1}

\section{Introduction}

\subsection{Outline}\label{subsection outline}

As the title and abstract suggest, this is essentially an omnibus paper in which we collect a number of useful results in $\infty$-category theory, all having to do in some way or another with the \bit{Grothendieck construction}.  This is instantiated in quasicategories by Lurie's \textit{unstraightening construction}, although we work model-independently (see \cref{subsection gr conventions}).

\begin{itemize}

\item In \cref{section gr}, we fix some notation and terminology surrounding the Grothendieck construction.  We also give some examples, and we highlight some of its important features -- notably its \textit{naturality} (as proved by Gepner--Haugseng--Nikolaus).

\item In \cref{section gr and colims}, we explore the relationship between the Grothendieck construction (and its ``two-sided'' generalization) and colimits in the $\infty$-category of spaces.  For instance, we record an $\infty$-categorical version of Thomason's \textit{homotopy colimit theorem}.  Some of these results are suggestive of the $(\infty,2)$-categorical functoriality of the Grothendieck construction (e.g.\! we say the word ``modification'').

\item In \cref{section op/lax}, we define \textit{lax} and \textit{oplax} natural transformations of functors $\C \ra \Cati$ via the Grothendieck construction.  Using these, we then construct a \bit{global colimit functor} for a cocomplete $\infty$-category $\C$: this is a functor
\[ \Lax(\C) \ra \C \]
from the \textit{lax overcategory} of $\C$
\begin{itemizesmall}
\item which sends an object
\[ (\D \xra{F} \C) \in \Lax(\C) \]
to its colimit
\[ \colim_\D(F) \in \C , \]
and
\item which encodes the simultaneous and interwoven functoriality of colimits in $\C$
\begin{itemizesmall}
\item for \textit{natural transformations} -- that is, for maps in $\Fun(\D,\C)$ -- and
\item for \textit{pullback along maps of diagram $\infty$-categories} -- that is, for maps in $(\Cati)_{/\C}$.
\end{itemizesmall}
\end{itemizesmall}
This immediately dualizes to give an analogous \bit{global limit functor}
\[ \opLax(\D)^{op} \xra{\lim} \D \]
for a complete $\infty$-category $\D$ (now running from the \textit{opposite} of its \textit{oplax} overcategory).

\item In \cref{section ho-p.b.'s in Thomason}, we prove $\infty$-categorical versions of Barwick--Kan's \bit{Theorem {\Bnbit}} and Dwyer--Kan--Smith's \bit{Theorem {\Cnbit}}, thus further extending the following sequence of increasingly general results in 1-category theory.
\begin{itemize}

\item Given a functor
$ \D \xra{F} \C $
satisfying a certain property B, Quillen's \textit{Theorem B} gives a simple description of the fibers
\[ \lim \left(  \begin{tikzcd}
& \D^\gpd \arrow{d}{F^\gpd} \\
\{ c \}^\gpd \arrow[hook]{r} & \C^\gpd
\end{tikzcd} \right) \]
of the induced map
on ($\infty$-)groupoid completions.

\item Given a functor satisfying a certain property {\Bn} (which recovers property B when $n=1$ but becomes weaker as $n$ grows),
\begin{itemizesmall}
\item Dwyer--Kan--Smith's \textit{Theorem {\Bn}} gives a description of the fibers of the induced map on groupoid completions (which recovers Quillen's description when $n=1$ but in trade becomes more complicated as $n$ grows), while
\item their \textit{Theorem {\Cn}} asserts that if $\C$ satisfies a certain property {\Cn}, then \textit{any} functor $\D \ra \C$ satisfies property {\Bn}.
\end{itemizesmall}

\item Given a cospan
$ \D \xra{F} \C \xla{G} \E $
such that $F$ satisfies property {\Bn} (but without any conditions on $G$), Barwick--Kan's \textit{enhanced} Theorem {\Bn} gives a description of the pullback
\[ \lim \left(  \begin{tikzcd}
& \D^\gpd \arrow{d}{F^\gpd} \\
\E^\gpd \arrow{r}[swap]{G^\gpd} & \C^\gpd
\end{tikzcd} \right) \]
of the induced cospan
on groupoid completions (which likewise becomes more complicated as $n$ grows).
\end{itemize}

\item In \cref{section bousfield--kan}, we prove a \bit{Bousfield--Kan formula} for colimits in $\infty$-categories.  This generalizes the 1-categorical formula for a colimit as a coequalizer of maps between coproducts.  We also illustrate its application with concrete examples.

\item In \cref{section thomason}, we construct a \bit{Thomason model structure} on the $\infty$-category $\Cati$ of $\infty$-categories.\footnote{See \cite[\sec 1]{MIC-sspaces} for the definition of a model structure on an $\infty$-category.}  Aside from its intrinsic interest, this model $\infty$-category $(\Cati)_\Thomason$ provides a convenient language for the various results which appear throughout this paper: it gives a presentation of the $\infty$-category $\S$ of spaces, and moreover its localization functor
\[ \Cati \ra \loc{\Cati}{\bW_\Thomason} \simeq \S \]
can be canonically identified with the groupoid completion functor.  In particular, the subcategory $\bW_\Thomason \subset \Cati$ of Thomason weak equivalences consists of precisely those maps which become equivalences upon groupoid completion.

This model structure is analogous to the classical Thomason model structure on the category $\strcat$ of categories; however, it is (in a sense) better behaved, and moreover it completely accounts for a certain quirk that prevents the latter from being lifted directly along the nerve functor.  On the other hand, this model structure bears some rather surprising features of its own.

\end{itemize}

\subsection{Conventions}\label{subsection gr conventions}

\partofMIC

\tableofcodenames

\examplecodename

\citelurie \ \luriecodenames

\butinvt \ \seeappendix

\subsection{Acknowledgments}  

We heartily thank David Ayala, David Gepner, Saul Glasman, and Zhen Lin Low for their helpful input on the material presented in this paper.  We are also grateful to the NSF graduate research fellowship program (grant DGE-1106400) for its financial support during the time that this work was carried out, and to the Pacific Science Institute for the hospitality and kite-surfing that it afforded during the time that this paper was being written.

\section{The Grothendieck construction}\label{section gr}

In this section, we recall some basic notions involving the Grothendieck construction and the various sorts of fibrations that it involves: we discuss co/cartesian (and left/right) fibrations in \cref{subsection fibns for gr}, and we discuss the Grothendieck construction itself in \cref{subsection gr}.  For background, we refer the reader to our companion paper \cite{grjl}, which contains
\begin{itemizesmall}
\item model-independent definitions of co/cartesian morphisms and co/cartesian fibrations,
\item proofs that these model-independent definitions are suitably compatible with their quasicategorical counterparts, and
\item an extended informal discussion the Grothendieck construction.
\end{itemizesmall}

\subsection{Fibrations}\label{subsection fibns for gr}

We begin by fixing the following notation (without really giving any definitions).

\begin{notn}\label{notation for grothendieck construction}
Let $\C$ be an $\infty$-category.
\begin{itemize}

\item We denote by $\LFib(\C)$ the $\infty$-category of \textit{left fibrations} over $\C$; by the dual of Corollary T.2.2.3.12 (see Remark T.2.1.4.12), this is the underlying $\infty$-category of the covariant model structure of Proposition T.2.1.4.7, and is well-defined by Remark T.2.1.4.11.

\item We denote by $\coCartFib(\C)$ the $\infty$-category of \textit{cocartesian fibrations} over $\C$; by the dual of Proposition T.3.1.4.1, this is the underlying $\infty$-category of the cocartesian model structure which is dual to that of Proposition T.3.1.3.7 (see Remark T.3.1.3.9), and is well-defined by Proposition T.3.3.1.1.

\end{itemize}
By Theorem T.3.1.5.1, these two $\infty$-categories sit in a sequence of adjunctions
\[ \begin{tikzcd}[column sep=2cm]
(\Cati)_{/\C} \arrow[transform canvas={yshift=0.7ex}]{r}{\leftloc_{\coCartFib(\C)}}[swap, transform canvas={yshift=0.2ex}]{\scriptstyle \bot} \arrow[transform canvas={yshift=-0.7ex}, hookleftarrow]{r}[swap]{\forget_{\coCartFib(\C)}}
& \coCartFib(\C)
{\arrow[transform canvas={yshift=0.7ex}]{r}{\leftloc_{\LFib(\C)}}[swap, transform canvas={yshift=0.2ex}]{\scriptstyle \bot} \arrow[transform canvas={yshift=-0.7ex}, hookleftarrow]{r}[swap]{\forget_{\LFib(\C)}}}
& \LFib(\C)
{\arrow[transform canvas={yshift=0.7ex}]{r}{\leftloc_{\L(\C)}}[swap, transform canvas={yshift=0.2ex}]{\scriptstyle \bot} \arrow[transform canvas={yshift=-0.7ex}, hookleftarrow]{r}[swap]{\forget_{\L(\C)}}}
& \S_{/\C^\gpd}
\end{tikzcd} \]
in $\Cati$.\footnote{The identification of $\S_{/\C^\gpd}$ as the underlying $\infty$-category of the corresponding model category follows from the fact that $s\Set_\KQ$ is right proper, and hence \textit{any} weak equivalence induces a Quillen adjunction on overcategories.  In particular, the overcategory of a quasicategory already has the correct underlying $\infty$-category, even without replacing the quasicategory by a Kan complex.}

Of course, there are dual notions of \textit{right fibrations} and of \textit{cartesian fibrations}, defined to be those functors that respectively become left fibrations or cocartesian fibrations under the involution $(-)^{op} : \Cati \ra \Cati$.  These then assemble into an analogous string of adjunctions
\[ \begin{tikzcd}[column sep=2cm]
(\Cati)_{/\C} \arrow[transform canvas={yshift=0.7ex}]{r}{\leftloc_{\CartFib(\C)}}[swap, transform canvas={yshift=0.2ex}]{\scriptstyle \bot} \arrow[transform canvas={yshift=-0.7ex}, hookleftarrow]{r}[swap]{\forget_{\CartFib(\C)}}
& \CartFib(\C)
{\arrow[transform canvas={yshift=0.7ex}]{r}{\leftloc_{\RFib(\C)}}[swap, transform canvas={yshift=0.2ex}]{\scriptstyle \bot} \arrow[transform canvas={yshift=-0.7ex}, hookleftarrow]{r}[swap]{\forget_{\RFib(\C)}}}
& \RFib(\C)
{\arrow[transform canvas={yshift=0.7ex}]{r}{\leftloc_{\R(\C)}}[swap, transform canvas={yshift=0.2ex}]{\scriptstyle \bot} \arrow[transform canvas={yshift=-0.7ex}, hookleftarrow]{r}[swap]{\forget_{\R(\C)}}}
& \S_{/\C^\gpd}
\end{tikzcd} \]
in $\Cati$.  By taking opposites, it will often suffice to leave observations about these latter notions implicit.
\end{notn}

We now assemble a number of useful observations.

\begin{rem}
The adjunctions $\leftloc_{\LFib(\C)} \adj \forget_{\LFib(\C)}$ and $\leftloc_{\L(\C)} \adj \forget_{\L(\C)}$ are both left localization adjunctions.\footnote{To see this, note that the first is presented by the composite of a left Bousfield localization followed by a Quillen equivalence, while the second is presented by a left Bousfield localization.}  However, the adjunction $\leftloc_{\coCartFib(\C)} \adj \forget_{\coCartFib(\C)}$ is not: given two objects $(\D \xra{p} \C) , (\E \xra{q} \C) \in \coCartFib(\C)$, a morphism
\[ \begin{tikzcd}
\D \arrow{rr} \arrow{rd}[swap]{p} & & \E \arrow{ld}{q} \\
& \C
\end{tikzcd} \]
in $(\Cati)_{/\C}$ must take $p$-cocartesian morphisms in $\D$ to $q$-cocartesian morphisms in $\E$ in order to define a morphism in $\coCartFib(\C)$.  (However, the right adjoint $\forget_{\coCartFib(\C)}$ is nevertheless the inclusion of a (non-full) subcategory, as indicated.)  In \cref{section op/lax}, we will see that its failure to be a left localization gives rise to the notion of a \textit{lax natural transformation} between objects of $\Fun(\C,\Cati)$.
\end{rem}

\begin{rem}
Given any functor $\D \xra{F} \C$, the resulting horizontal composite in the diagram
\[ \begin{tikzcd}
\D \times_\C \Fun([1],\C) \arrow{r} \arrow{d} & \Fun([1],\C) \arrow{r}{t} \arrow{d}{s} & \C \\
\D \arrow{r}[swap]{F} & \C
\end{tikzcd} \]
is a cocartesian fibration by (the dual of) Corollary T.2.4.7.12.  In fact, by (the dual of) \cite[Theorem 4.5]{GHN}, this is precisely the \textit{free cocartesian fibration} on $F$: this construction gives the left adjoint $\leftloc_{\coCartFib(\C)}$.
\end{rem}

\begin{rem}\label{always groupoid-complete}
In the sequence
\[ (\Cati)_{/\C} \xra{\leftloc_{\coCartFib(\C)}} \coCartFib(\C) \xra{\leftloc_{\LFib(\C)}} \LFib(\C) \xra{\leftloc_{\L(\C)}} \S_{\C^\gpd} \]
of left adjoints, all of the (possibly composite) left adjoints with target $\S_{/\C^\gpd}$ are given by taking an object $\D \ra \C$ to the object $\D^\gpd \ra \C^\gpd$; this follows directly from the definitions of the overlying Quillen adjunctions of Theorem T.3.1.5.1.
\end{rem}

\begin{rem}\label{over a space Gr two adjns become equivces}
When in fact $\C \in \S \subset \Cati$, the adjunctions $\leftloc_{\coCartFib(\C)} \adj \forget_{\coCartFib(\C)}$ and $\leftloc_{\L(\C)} \adj \forget_{\L(\C)}$ become adjoint equivalences by Theorem T.3.1.5.1.
\end{rem}

\subsection{The Grothendieck construction}\label{subsection gr}

We now give the main definition of this section.

\begin{defn}\label{define grothendieck construction}
We define the \bit{Grothendieck construction} to be either equivalence of $\infty$-categories in the diagram
\[ \begin{tikzcd}
\Fun(\C,\Cati) \arrow{r}{\Gr}[swap]{\sim} & \coCartFib(\C) \\
\Fun(\C,\S) \arrow[hook]{u} \arrow{r}{\sim}[swap]{\Gr} & \LFib(\C) ; \arrow[hook]{u}
\end{tikzcd} \]
here, the upper (resp.\! lower) equivalence underlies the right adjoint of the Quillen equivalence which is dual to that of Theorem T.3.2.0.1 (resp.\! that of Theorem T.2.2.1.2) in the special case that the functor of $s\Set$-enriched categories is the identity.  The fact that the diagram commutes follows from the construction (see Definition T.3.2.1.2).  Of course, there are also Grothendieck constructions for cartesian fibrations and right fibrations; we will denote these by
\[ \begin{tikzcd}
\Fun(\C^{op},\Cati) \arrow{r}{\Grop}[swap]{\sim} & \CartFib(\C) \\
\Fun(\C^{op},\S) \arrow[hook]{u} \arrow{r}{\sim}[swap]{\Grop} & \RFib(\C) . \arrow[hook]{u}
\end{tikzcd} \]
When we need to distinguish between these two types of Grothendieck constructions, we will refer to the former sort as \bit{covariant} and to the latter sort as \bit{contravariant}.

For any $\C \xra{F} \Cati$, when we need to refer to it we will write
\[ \Gr(F) \xra{\pr_{\Gr(F)}} \C \]
for the cocartesian fibration that it classifies.  Similarly, given any $\C^{op} \xra{G} \Cati$, when we need to refer to it we will write
\[ \Grop(G) \xra{\pr_{\Grop(G)}} \C \]
for the cartesian fibration that it classifies.
\end{defn}

The following characterization of the Grothendieck construction may at first appear rather abstract.  However, it gives excellent geometric intuition, as we illustrate in \cref{ex cocart over walking arrow} (the first nontrivial case).

\begin{rem}
The covariant Grothendieck construction can be characterized as a \textit{lax colimit}: by \cite[Theorem 7.4]{GHN}, for any functor $\C \xra{F} \Cati$ we have an equivalence
\[ \Gr(F) \simeq \int^{c \in \C} \C_{c/} \times F(c) \]
of its covariant Grothendieck construction with its colimit weighted by $\C^{op} \xra{\C_{-/}} \Cati$.  Thus, whereas the ordinary colimit of the diagram $F$ can be viewed as ``gluing together'' all of the $\infty$-categories $F(c)$ -- that is, taking all of the values $F(c) \in \Cati$ and, for every map $c \xra{\varphi} c'$ in $\C$ adding in an \textit{equivalence}
\[ y \xra{\sim} (F(\varphi))(y) \]
for every $y \in F(c)$ --, in a \textit{lax} colimit, we now only add in a \textit{noninvertible} morphism
\[ y \ra (F(\varphi))(y) \]
corresponding to such data.  Dually, the contravariant Grothendieck construction can be characterized as an \textit{oplax colimit}: by \cite[Corollary 7.6]{GHN}, for any functor $\C^{op} \xra{G} \Cati$ we have an equivalence
\[ \Grop(G) \simeq \int^{c^\opobj \in \C^{op}} \C_{/c} \times G(c^\opobj) \]
of its contravariant Grothendieck construction with its colimit weighted by $\C \xra{\C_{/-}} \Cati$.\footnote{One can also identify the \textit{$\infty$-categories of sections} of Grothendieck constructions with op/lax \textit{limits} (see e.g.\! \cite[Proposition 7.1]{GHN}), but this is less essential for geometric intuition.}
\end{rem}

\begin{ex}\label{ex cocart over walking arrow}
Suppose that $[1] \xra{F} \Cati$ selects a functor $\C_0 \xra{f} \C_1$.  Then, its covariant Grothendieck construction can be identified as
\[ \Gr(F) \simeq \colim \left( \begin{tikzcd}
\C_0 \arrow{r}{f} \arrow{d}[swap]{\id_{\C_0} \times \{1\}} & \C_1 \\
\C_0 \times [1]
\end{tikzcd} \right) , \]
a ``directed mapping cylinder'' for $f$.  Dually, if $[1]^{op} \xra{G} \Cati$ selects a functor $\D_0 \xla{g} \D_1$, then its contravariant Grothendieck construction can be identified as
\[ \Grop(G) \simeq \colim \left( \begin{tikzcd}
\D_1 \arrow{d}[swap]{\id_{\D_1} \times \{0\}} \arrow{r}{g} & \D_0 \\
\D_1 \times [1]
\end{tikzcd} \right) , \]
a ``reversed directed mapping cylinder'' for $g$.
\end{ex}

We now list a few more basic examples of the Grothendieck construction.

\begin{ex}\label{Gr preserves terminal objects}
The equivalence
\[ \Fun(\C,\Cati) \xra[\sim]{\Gr} \coCartFib(\C) \]
necessarily preserves terminal objects.  In the source, the terminal object is $\const(\pt_{\Cati})$, while in the target, the terminal object is the identity functor $\id_\C$ (which is a cocartesian fibration).  Similarly, the identity functor $\id_\C$ is the terminal object of $\CartFib(\C)$.
\end{ex}

\begin{ex}\label{fiber prods of cocart fibns is Gr of product}
Given any two functors $F,G \in \Fun( \C , \Cati)$, consider the pullback diagram
\[ \begin{tikzcd}
\Gr(F) \times_\C \Gr(G) \arrow{r} \arrow{d} & \Gr(F) \arrow{d}{\pr_{\Gr(F)}} \\
\Gr(G) \arrow{r}[swap]{\pr_{\Gr(G)}} & \C
\end{tikzcd} \]
in $\Cati$.  Note first that the composite functor
\[ \Gr(F) \times_\C \Gr(G) \ra \C \]
is again a cocartesian fibration by the dual of (parts (2) and (3) of) Proposition T.2.4.2.3.  Moreover, this is the product of the objects $\pr_{\Gr(F)}$ and $\pr_{\Gr(G)}$ in $(\Cati)_{/\C}$, and since the inclusion $\coCartFib(\C) \subset (\Cati)_{/\C}$ is a right adjoint, this must also be their product in $\coCartFib(\C)$.  Thus, this cocartesian fibration must be classified by the composite functor
\[ \C \xra{(F,G)} \Cati \times \Cati \xra{- \times -} \Cati , \]
i.e.\! the product $F \times G \in \Fun(\C,\Cati)$.
\end{ex}

\begin{ex}\label{Gr is trivial over pt}
In the special case that $\C = \pt_\Cati$, the Grothendieck construction yields an equivalence
\[ \Cati \simeq \Fun(\pt_{\Cati},\Cati) \xra[\sim]{\Gr} \coCartFib(\pt_{\Cati}) . \]
By \cite[Th\'{e}or\`{e}m 6.3]{Toen}, this must be inverse to the composite
\[ \coCartFib(\pt_{\Cati}) \xra{\sim} (\Cati)_{/\pt_{\Cati}} \xra{\sim} \Cati \]
of two forgetful equivalences.
\end{ex}

\begin{ex}\label{slice cats are l/r fibns}
Let $\C \in \Cati$, and let $c \in \C$.  Then, the forgetful functor $\C_{c/} \ra \C$ from the undercategory is a left fibration by Corollary T.2.1.2.2; in fact, it follows from Proposition T.4.4.4.5 that we can identify this as
\[ \C_{c/} \simeq \Gr \left( \C \xra{\hom_\C(c,-)} \S \right) . \]
Dually, we have that $(\C_{/c} \ra \C) \in \RFib(\C)$, and we have the identification
\[ \C_{/c} \simeq \Grop \left( \C^{op} \xra{\hom_\C(-,c)} \S \right) . \]
\end{ex}

\begin{rem}\label{naturality of Gr}
An important property of the Grothendieck construction is its \textit{naturality}: it assembles to a functor
\[ (\Cati)^{op} \ra \Fun([1],\Cati) \]
which sends an $\infty$-category $\C$ to the object
\[ \Fun(\C,\Cati) \xra[\sim]{\Gr} \coCartFib(\C) \]
of $\Fun([1],\Cati)$ by (the dual of) \cite[Corollary A.31]{GHN}.  By its construction (see the proof of \cite[Proposition A.30]{GHN}), over $0 \in [1]$ this functor is given by precomposition of functors to $\Cati$, while over $1 \in [1]$ this functor is given by pullback of cocartesian fibrations.  So for instance, a map $\C \xra{F} \D$ in $\Cati$ determines a map
\[ \begin{tikzcd}
\Fun(\D , \Cati) \arrow{r}{\Gr}[swap]{\sim} \arrow{d}[swap]{- \circ F} & \coCartFib(\D) \arrow{d}{F^*} \\
\Fun(\C , \Cati) \arrow{r}{\sim}[swap]{\Gr} & \coCartFib(\C)
\end{tikzcd} \]
in $\Fun([1],\Cati)$ (where we read this square in $\Cati$ vertically in order to consider it as a morphism between the two horizontal arrows).  Of course, by duality the contravariant Grothendieck construction enjoys analogous naturality.
\end{rem}

Since this observation will arise so frequently for us, we codify it.

\begin{defn}
We refer to the phenomenon described in \cref{naturality of Gr} as the \bit{naturality} of the Grothendieck construction.
\end{defn}

This has the following easy consequence.

\begin{ex}\label{Gr of a constant functor}
Consider the tautological factorization
\[ \begin{tikzcd}
\C \arrow{rr}{\const(\D)} \arrow[dashed]{rd} & & \Cati . \\
& \pt_\Cati \arrow{ru}[swap]{\D}
\end{tikzcd} \]
By \cref{Gr is trivial over pt}, we have a canonical equivalence
\[ \Gr \left( \pt_\Cati \xra{\D} \Cati \right) \simeq \D . \]
Hence, the naturality of the Grothendieck construction implies that the pullback square
\[ \begin{tikzcd}
\C \times \D \arrow{r} \arrow{d} & \D \arrow{d} \\
\C \arrow{r} & \pt_\Cati
\end{tikzcd} \]
in $\Cati$ provides a canonical equivalence
\[ \Gr \left( \C \xra{\const(\D)} \Cati \right) \simeq \C \times \D , \]
with the cocartesian fibration down to $\C$ identifying as $\pr_{\Gr(\const(\D))} \simeq \pr_\C$, the projection onto the first factor.  (From here, we can recover \cref{Gr preserves terminal objects} as the special case where $\D = \pt_\Cati$.)  Similarly, we have that
\[ \Grop \left( \C^{op} \xra{\const(\D)} \Cati \right) \simeq \C \times \D , \]
with $\pr_{\Grop(\const(\D))} \simeq \pr_\C$: in other words, the projection $\C \times \D \xra{\pr_\C} \C$ is simultaneously a cocartesian fibration and a cartesian fibration, in either case classified by the constant functor at the object $\D \in \Cati$.
\end{ex}

\begin{defn}\label{def fib}
Fix some $\C \xra{F} \Cati$.  By the naturality of the Grothendieck construction (and \cref{Gr is trivial over pt}), for any $x \in \C$ there is a canonical pullback square
\[ \begin{tikzcd}
F(x) \arrow{r} \arrow{d} & \Gr(F) \arrow{d} \\
\{ x \} \arrow[hook]{r} & \C
\end{tikzcd} \]
in $\Cati$.  We refer to $F(x)$ as the \bit{fiber} of the cocartesian fibration over the object $x \in \C$, and we refer to the above commutative square as a \bit{fiber inclusion}.\footnote{At the level of quasicategories, this can be computed by the inclusion of the fiber over a vertex corresponding to $x \in \C$, which is a homotopy fiber in $s\Set_\Joyal$ by the Reedy trick (and the implications of Remark T.2.0.0.5).}
\end{defn}

\begin{rem}
A fiber inclusion will not generally be the inclusion of a full subcategory: it will not contain those morphisms covering nontrivial endomorphisms.  In fact, in the extreme case that we take $\C = Y \in \S \subset \Cati$ to be an $\infty$-groupoid, the functor corepresented by an object $y \in Y$ will have
\[ \Gr \left( Y \xra{\hom_Y(y,-)} \S \right) \simeq Y_{y/} \simeq \pt_\S \]
(recall \cref{slice cats are l/r fibns}); then, the fiber inclusion over the object $y \in Y$ itself will be given by the pullback square
\[ \begin{tikzcd}
\Omega Y \arrow{r} \arrow{d} & \pt_\S \arrow{d} \\
\{ y \} \arrow[hook]{r} & Y
\end{tikzcd} \]
in $\S \subset \Cati$, the ``inclusion'' of the based loopspace of $Y$ at $y$ into the terminal space.
\end{rem}

We have the following concrete identification of the left localization which takes cocartesian fibrations to left fibrations.

\begin{prop}\label{cocartfib to lfib corresponds to groupoid-completion}
Under the equivalences given by the Grothendieck construction, the left localization
\[ \coCartFib(\C) \xra{\leftloc_{\LFib(\C)}} \LFib(\C) \]
corresponds to the functor $\Fun(\C,\Cati) \ra \Fun(\C,\S)$ given by postcomposition with
\[ \Cati \xra{({-})^\gpd} \S . \]
\end{prop}

\begin{proof}
This follows from the uniqueness of left adjoints and the commutativity of the diagram in \cref{define grothendieck construction}: the adjunction $\Fun(\C,\Cati) \adjarr \Fun(\C,\S)$ comes from applying the functor $\Fun(\C,-) : \Cati \ra \Cati$ to the adjunction $({-})^\gpd : \Cati \adjarr \S$.
\end{proof}

\section{The Grothendieck construction and colimits of spaces}\label{section gr and colims}

In this section, we study the relationship between the Grothendieck construction and colimits in the $\infty$-category $\S$ of spaces: in \cref{subsection gr and colims} we give some basic results, in \cref{subsection two-sided gr and colims} we extend these to the ``two-sided'' Grothendieck construction, and in \cref{subsection two-categorical gr and colims} we collect some results which are suggestive of the $(\infty,2)$-categorical functoriality of the Grothendieck construction.

\subsection{The Grothendieck construction and colimits}\label{subsection gr and colims}

We begin with the following fundamental fact, on which all further results in this direction are based.  Its 1-categorical version, Thomason's \textit{homotopy colimit theorem}, first appeared as \cite[Theorem 1.2]{Thomasonhocolim}.

\begin{prop}\label{groupoid-completion of the grothendieck construction}
For any $\C \xra{F} \Cati$, the Grothendieck construction computes its homotopy colimit when considered as diagram in $(\Cati)_\Thomason$, i.e.\!
\[ \Gr(F)^\gpd \simeq \colim \left( \C \xra{F} \Cati \xra{({-})^\gpd} \S \right) . \]
\end{prop}

\begin{proof}
This follows from Corollary T.3.3.4.3 and the fact that groupoid completion (being a left adjoint) commutes with colimits.
\end{proof}

\begin{cor}\label{natural thomason weak equivalence induces a thomason weak equivalence on grothendieck constructions}
Suppose that we are given a pair of functors $F,G : \C \rra (\Cati)_\Thomason$ and a natural weak equivalence $F \we G$.  Then the induced map $\Gr(F) \ra \Gr(G)$ is a weak equivalence in $(\Cati)_\Thomason$.
\end{cor}

\begin{proof}
Combining the equivalence
\[ (-)^\gpd \circ F \xra{\sim} (-)^\gpd \circ G \]
in $\Fun(\C,\S)$ with two applications of the equivalence of \cref{groupoid-completion of the grothendieck construction}, we obtain a composite equivalence
\[ \Gr(F_1)^\gpd \simeq \colim \left( \C \xra{F_1} \Cati \xra{({-})^\gpd} \S \right) \xra{\sim} \colim \left( \C \xra{F_2} \Cati \xra{({-})^\gpd} \S \right) \simeq \Gr(F_2)^\gpd \]
in $\S$.
\end{proof}

\subsection{The two-sided Grothendieck construction and colimits}\label{subsection two-sided gr and colims}

We will also be interested in the following variant of the Grothendieck construction and its interaction with colimits in $\S$.

\begin{defn}\label{define two-sided Gr}
Given any (ordered) pair of functors $\C^{op} \xra{F} \Cati$ and $\C \xra{G} \Cati$, we define the \bit{two-sided Grothendieck construction} of $F$ and $G$ to be the pullback
\[ \Gr(F,\C,G) = \lim \left( \begin{tikzcd}
& \Grop(F) \arrow{d} \\
\Gr(G) \arrow{r} & \C
\end{tikzcd} \right) \]
in $\Cati$.
\end{defn}

\begin{prop}\label{invce of two-sided Gr}
Suppose that we are given
\begin{itemizesmall}
\item a pair of functors $F,F' : \C^{op} \rra (\Cati)_\Thomason$ and a natural weak equivalence $F \we F'$, and
\item a pair of functors $G,G' : \C \rra (\Cati)_\Thomason$ and a natural weak equivalence $G \we G'$.
\end{itemizesmall}
Then, the induced map
\[ \Gr(F,\C,G) \ra \Gr(F',\C,G') \]
is a weak equivalence in $(\Cati)_\Thomason$.
\end{prop}

\begin{proof}
The naturality of the Grothendieck construction induces an evident naturality of the two-sided Grothendieck construction; in light of this, by duality and since $\bW_\Thomason \subset \Cati$ is closed under composition, it suffices to prove the claim in the special case that $G = G'$ and that the natural weak equivalence $G \we G'$ is simply $\id_G$.  Then, by the naturality of the Grothendieck construction we have an equivalence
\[ \Gr(F,\C,G) = \lim \left( \begin{tikzcd}
& \Grop(F) \arrow{d} \\
\Gr(G) \arrow{r} & \C
\end{tikzcd} \right) \simeq \Grop(F \circ \pr_{\Gr(G)}^{op}) \]
in $\CartFib(\Gr(G))$, and similarly we have an equivalence
\[ \Gr(F',\C,G) \simeq \Grop(F' \circ \pr_{\Gr(G)}^{op}) \]
in $\CartFib(\Gr(G))$.  Moreover, by assumption, the map
\[ F \circ \pr_{\Gr(G)}^{op} \ra F' \circ \pr_{\Gr(G)}^{op} \]
in $\Fun(\Gr(G)^{op} , \Cati)$ has that for every $y \in \Gr(G)^{op}$, the induced map
\[ (F \circ \pr_{\Gr(G)}^{op})(y)^\gpd \ra (F' \circ \pr_{\Gr(G)}^{op})(y)^\gpd \]
is an equivalence in $\S$: in other words, the map
\[ (-)^\gpd \circ F \circ \pr_{\Gr(G)}^{op} \ra (-)^\gpd \circ F' \circ \pr_{\Gr(G)}^{op} \]
is an equivalence in $\Fun(\Gr(G)^{op},\S)$.  The claim now follows from (the dual of) \cref{natural thomason weak equivalence induces a thomason weak equivalence on grothendieck constructions}.
\end{proof}

\subsection{Higher-categorical functoriality of the Grothendieck construction and colimits}\label{subsection two-categorical gr and colims}

We now assemble a few results which are suggestive of the $(\infty,2)$-categorical functoriality of the Grothendieck construction.  However, we do not pursue such functoriality in any systematic way.

\begin{prop}\label{triangle of colimits}
A diagram
\[ \begin{tikzcd}
\C \arrow[bend left=50]{r}{E}[swap, pos=0.1, transform canvas={yshift=-0.4em}]{\left. \alpha \right\Downarrow} \arrow[bend right=50]{r}[swap]{F} & \D \arrow{r}{G} & \Cati
\end{tikzcd} \]
induces a commutative triangle
\[ \begin{tikzcd}[row sep=0.5cm]
\colim_\C((-)^\gpd \circ G \circ E) \arrow{dd} \arrow{rd} \\
& \colim_\D((-)^\gpd \circ G) \\
\colim_\C((-)^\gpd \circ G \circ F) \arrow{ru}
\end{tikzcd} \]
in $\S$.
\end{prop}

\begin{proof}
This follows by combining \cref{groupoid-completion of the grothendieck construction}, \cref{lax triangle}, and \cref{rnerves:nat trans induces equivce betw maps on gpd-complns}.
\end{proof}

The following result, an ingredient of the proof of \cref{triangle of colimits}, uses the language of \cref{subsection op/lax nat trans}.

\begin{lem}\label{lax triangle}
A diagram
\[ \begin{tikzcd}
\C \arrow[bend left=50]{r}{E}[swap, pos=0.1, transform canvas={yshift=-0.4em}]{\left. \alpha \right\Downarrow} \arrow[bend right=50]{r}[swap]{F} & \D \arrow{r}{G} & \Cati
\end{tikzcd} \]
induces a morphism
\[ \begin{tikzcd}[row sep=0.5cm]
\Gr(G \circ E) \arrow{dd}[swap]{\Gr(\id_G \circ \alpha)} \arrow{rd}[swap, sloped, pos=0.25, transform canvas={yshift=-1.5em}]{\Downarrow} \\
& \Gr(G) \\
\Gr(G \circ F) \arrow{ru}
\end{tikzcd} \]
in $\Lax(\Gr(G))$.
\end{lem}

\begin{proof}
We begin by encoding $\alpha$ as a functor $[1] \times \C \xra{H} \D$.  By the naturality of the Grothendieck construction, this gives rise to a diagram
\[ \begin{tikzcd}
\Gr(G \circ H) \arrow{r} \arrow{d} & \Gr(G) \arrow{d} \\
{[1] \times \C} \arrow{r}[swap]{H} \arrow{d}[swap]{\pr_{[1]}} & \D \\
{[1]}
\end{tikzcd} \]
in which both left vertical arrows are cocartesian fibrations (the upper by pullback, the lower the Grothendieck construction of the functor $[1] \xra{\const(\C)} \Cati$).  Hence, the composite left vertical arrow is a cocartesian fibration as well, namely the one classified by the map $[1] \ra \Cati$ selecting the functor
\[ \Gr(G \circ E) \xra{\Gr(\id_G \circ \alpha)} \Gr(G \circ F) . \]

Now, in order to extract this functor as a map \textit{between} cocartesian fibrations (as opposed to being contained \textit{within} a cocartesian fibration), we obtain a morphism in $\Fun([1], \Cati)$ with this map as its target by precomposing with the map $[1] \times [1] \ra [1]$ given by
\[ (i,j) \mapsto \left\{ \begin{array}{ll} 0 , & (i,j) \not= (1,1) \\ 1 , & (i,j) = (1,1) ; \end{array} \right. \]
by adjunction, this yields a functor $[1] \ra \Fun([1],\Cati)$ which, considered as a map in $\Fun([1],\Cati)$, has source $[1] \xra{\const(\Gr(G \circ E))} \Cati$ and has target selecting this same functor $\Gr(G \circ E) \xra{\Gr(\id_G \circ \alpha)} \Gr(G \circ F)$.  From here, the equivalence $\Fun([1],\Cati) \xra[\sim]{\Gr} \coCartFib([1])$ gives rise to the diagram of \cref{diagram in proof of lax triangle}, in which for clarity we include both fiber inclusions of each of these objects of $\coCartFib([1])$ as well as the induced maps between them.
\begin{figure}[h]
\[ \begin{tikzcd}[column sep=0.5cm]
& & \Gr(G \circ E) \arrow{rrrr}{\Gr(\id_G \circ \alpha)} \arrow{ld} \arrow{ddddddrr} & & & & \Gr(G \circ F) \arrow{ld} \arrow{ddddddll} \\
& {[1] \times \Gr(G\circ E)} \arrow[crossing over]{rrrr} \arrow{ddddddrr} & & & & \Gr(G \circ H) \arrow[crossing over]{rrrr} \arrow[crossing over]{ddddddll} & & & & \Gr(G) \\
\Gr(G \circ E) \arrow[crossing over]{rrrr}{\id_{\Gr(G \circ E)}} \arrow{ru} \arrow{ddddddrr} & & & & \Gr(G \circ E) \arrow{ru} \arrow[crossing over]{ddddddll} \\ \\ \\ \\
& & & & \{1\} \\
& & & {[1]} \arrow[hookleftarrow]{ru} \\
& & \{0\} \arrow[hookrightarrow]{ru}
\end{tikzcd} \]
\caption{The diagram in $\Cati$ used in the proof of \cref{lax triangle}.}
\label{diagram in proof of lax triangle}
\end{figure}
Moreover, there is a canonical natural transformation in $\Fun(\Gr(G \circ E) , [1] \times \Gr(G \circ E))$ from the inclusion of the fiber over $0 \in [1]$ to the inclusion of the fiber over $1 \in [1]$ (selected by the identity map $\id_{[1] \times \Gr(G \circ E)}$).  Taking the horizontal composite
\[ \begin{tikzcd}
\Gr(G \circ E) \arrow[bend left=50]{r}{\{0\} \times \id_{\Gr(G \circ E)}}[swap, transform canvas={yshift=-2em}]{\Downarrow} \arrow[bend right=50]{r}[swap]{\{1\} \times \id_{\Gr(G \circ E)}} & {[1] \times \Gr(G \circ E)} \arrow{r} & \Gr(G \circ H) \arrow{r} & \Gr(G)
\end{tikzcd} \]
of this natural transformation with the horizontal composite in the diagram of \cref{diagram in proof of lax triangle} then yields the desired morphism in $\Lax(\Gr(G))$.
\end{proof}

In manipulating colimits, we will also make use of the following notion (which is actually only a special case of a more general $(\infty,2)$-categorical phenomenon).

\begin{defn}
Let $\C \in \Cati$, let $F,G \in \Fun(\C,\Cati)$, and let $\alpha,\beta \in \hom_{\Fun(\C,\Cati)}(F,G)$.  A \bit{modification} from $\alpha$ to $\beta$ is a map
\[ \const([1]) \times F \xra{\mu} G \]
in $\Fun(\C,\Cati)$ whose restriction along $\const(\{0\}) \ra \const([1])$ recovers $\alpha$ and whose restriction along $\const(\{1\}) \ra \const([1])$ recovers $\beta$.
\end{defn}

The following result describes the effect of the Grothendieck construction on modifications.

\begin{prop}\label{mod gives nat trans}
Let $\C \in \Cati$, let $F,G \in \Fun(\C,\Cati)$, and let $\alpha,\beta \in \hom_{\Fun(\C,\Cati)}(F,G)$.  A modification $\mu$ from $\alpha$ to $\beta$ induces a natural transformation
\[ \begin{tikzcd}[column sep=2cm]
\Gr(F) \arrow[bend left=50]{r}{\Gr(\alpha)}[swap, pos=0.17, transform canvas={yshift=-1.2em}]{\left. \Gr(\mu) \right\Downarrow} \arrow[bend right=50]{r}[swap]{\Gr(\beta)} & \Gr(G)
\end{tikzcd} \]
in $\Cati$.
\end{prop}

\begin{proof}
Applying the Grothendieck construction to the modification $\mu$, we obtain a map
\[ \Gr(\const([1]) \times F) \xra{\Gr(\mu)} \Gr(G) \]
in $\coCartFib(\C)$.  But this source can be identified as
\[ \Gr(\const([1]) \times F) \simeq \Gr(\const([1])) \times_\C \Gr(F) \simeq ([1] \times \C) \times_\C \Gr(F) \simeq [1] \times \Gr(F) , \]
where the first equivalence follows from the fact that $\Fun(\C,\Cati) \xra[\sim]{\Gr} \coCartFib(\C)$ is an equivalence (so commutes with products) and the fact that the forgetful functor $\coCartFib(\C) \ra (\Cati)_{/\C}$ is a right adjoint and hence commutes with products.  So, this becomes a map
\[ [1] \times \Gr(F) \xra{\Gr(\mu)} \Gr(G) , \]
as desired.
\end{proof}

\section{Op/lax natural transformations and the global co/limit functor}\label{section op/lax}

In ordinary category theory, functors with the same source and target can be related by natural transformations between them.  When the target (and possibly the source) is a $2$-category, this notion can be relaxed in two different ways, yielding notions of \textit{lax} and \textit{oplax} natural transformations.

We will be interested in this phenomenon in the $\infty$-categorical context.  However, we will concern ourselves exclusively with the special case in which the source is only an $\infty$-category and the target is $\Cati$, considered as an $(\infty,2)$-category via the closure of its symmetric monoidal structure $(\Cati,\times)$.  In this case, the Grothendieck construction allows us to easily and concisely define such transformations without reference to an ambient theory of $(\infty,2)$-categories: heuristically speaking, it ``reduces category level by one''.

In \cref{subsection op/lax nat trans} we define op/lax natural transformations via the Grothendieck construction, and then in \cref{subsection global co/limit} we apply this framework to study the ($(\infty,2)$-categorical) functoriality of colimits in an arbitrary but fixed $\infty$-category.

\subsection{Op/lax natural transformations}\label{subsection op/lax nat trans}

We begin with the following definition.

\begin{defn}\label{define op/lax nat trans betw fctrs to Cati}
Suppose we are given a pair of functors $F,G : \C \rra \Cati$.  A \bit{lax natural transformation} from $F$ to $G$, denoted $F \laxra G$ for short, is a map
\[ \begin{tikzcd}
\Gr(F) \arrow{rd} \arrow{rr} & & \Gr(G) \arrow{ld} \\
& \C
\end{tikzcd} \]
in $(\Cati)_{/\C}$.  Meanwhile, an \bit{oplax natural transformation} from $F$ to $G$, denoted $F \oplaxra G$ for short, is a map
\[ \begin{tikzcd}
\Grop(F) \arrow{rd} \arrow{rr} & & \Grop(G) \arrow{ld} \\
& \C^{op}
\end{tikzcd} \]
in $(\Cati)_{/\C^{op}}$.
\end{defn}

To provide some intuition, we illustrate \cref{define op/lax nat trans betw fctrs to Cati} in the simplest nontrivial case.

\begin{ex}\label{ex lax from 1}
Recall from \cref{ex cocart over walking arrow} that we can think of the equivalence
\[ \Fun([1],\Cati) \xra[\sim]{\Gr} \coCartFib([1]) \]
as a sort of ``directed mapping cylinder'' construction.  Hence, if $[1] \xra{F} \Cati$ selects a functor $\C_0 \xra{f} \C_1$ and $[1] \xra{G} \Cati$ selects a functor $\D_0 \xra{g} \D_1$, then a lax natural transformation $\alpha : F \laxra G$ is equivalent to a square
\[ \begin{tikzcd}
\C_0 \arrow{r}{f}[swap, transform canvas={yshift=-1.2em}]{\RUarrow} \arrow{d}[swap]{\alpha_0} & \C_1 \arrow{d}{\alpha_1} \\
\D_0 \arrow{r}[swap]{g} & \D_1
\end{tikzcd} \]
in $\Cati$, i.e.\! the data of
\begin{itemize}
\item a functor $\C_0 \xra{\alpha_0} \D_0$,
\item a functor $\C_1 \xra{\alpha_1} \D_1$, and
\item a natural transformation $g \circ \alpha_0 \ra \alpha_1 \circ f$ in $\Fun(\C_0,\D_1)$.
\end{itemize}
A similar analysis shows that an oplax natural transformation $\alpha : F \oplaxra G$ simply reverses the direction of the natural transformation (so that it runs down and to the left, instead of up and to the right).
\end{ex}

We will be interested in the special case of \cref{ex lax from 1} in which $G = \const(\C)$ (so that $g = \id_\C$) for some chosen $\infty$-category $\C$.  For a fixed such choice, these can be organized in the following way.

\begin{defn}\label{define op/lax overcat}
For any $\C \in \Cati$, we define the \bit{lax overcategory} of $\C$ to be
\[ \Lax(\C) = \Grop \left( (\Cati)^{op} \xra{\Fun(-,\C)} \Cati \right) . \]
Thus, the objects of $\Lax(\C)$ are functors with target $\C$, and a morphism from $\D \xra{F} \C$ to $\E \xra{G} \C$ in $\Lax(\C)$ is given by a triangle
\[ \begin{tikzcd}
\D \arrow{rr}{H}[swap, transform canvas={yshift=-0.8em}]{\rotatebox[origin=c]{20}{$\Rightarrow$}} \arrow{rd}[swap]{F} & & \E \arrow{ld}{G} \\
& \C
\end{tikzcd} \]
in $\Cati$.  We write
\[ \Lax(\C) \xra{s_{\Lax(\C)}} \Cati \]
for the canonical projection map, and refer to it as the \bit{source projection}; this is by definition a cartesian fibration, with fiber over $\D \in \Cati$ given by $\Fun(\D,\C)$.

Similarly, we define the \bit{oplax overcategory} of $\C$ to be
\[ \opLax(\C) = \Grop \left( (\Cati)^{op} \xra{\Fun((-)^{op},\C^{op})} \Cati \right) . \]
Thus, the objects of $\opLax(\C)$ can be canonically identified (via the involution $(-)^{op} : \Cati \xra{\sim} \Cati$) with functors with target $\C$, and a morphism from (the object identified with) $\D \xra{F} \C$ to (the object identified with) $\E \xra{G} \C$ in $\opLax(\C)$ can be canonically identified with a triangle
\[ \begin{tikzcd}
\D \arrow{rr}{H}[swap, transform canvas={yshift=-0.8em}]{\rotatebox[origin=c]{200}{$\Rightarrow$}} \arrow{rd}[swap]{F} & & \E \arrow{ld}{G} \\
& \C
\end{tikzcd} \]
in $\Cati$.  We write
\[ \opLax(\C) \xra{s_{\opLax(\C)}} \Cati \]
for the canonical projection map, and also refer to it as the \bit{source projection}; this is again by definition a cartesian fibration, with fiber over $\D \in \Cati$ given by $\Fun(\D^{op},\C^{op}) \simeq \Fun(\D,\C)^{op}$.
\end{defn}

\begin{rem}
A map $\C_1 \ra \C_2$ induces a natural transformation $\Fun(-,\C_1) \ra \Fun(-,\C_2)$ in $\Fun((\Cati)^{op},\Cati)$, which in turn gives rise to a commutative triangle
\[ \begin{tikzcd}
\Lax(\C_1) \arrow{rr} \arrow{rd}[swap]{s_{\Lax(\C_1)}} & & \Lax(\C_2) \arrow{ld}{s_{\Lax(\C_2)}} \\
& \Cati
\end{tikzcd} \]
in $\Cati$; altogether, we obtain a functor
\[ \Cati \xra{\Lax(-)} \CartFib(\Cati) . \]
Similarly, we obtain a functor
\[ \Cati \xra{\opLax(-)} \CartFib(\Cati) . \]
\end{rem}

\begin{rem}\label{oplax of C is lax of Cop}
Expanding out the definition, we also can write
\[ \opLax(\C) = \Grop \left( (\Cati)^{op} \xra{((-)^{op})^{op}} (\Cati)^{op} \xra{\Fun(-,\C^{op})} \Cati \right) \]
(where the first functor is obtained by applying the involution $\Cati \xra{(-)^{op}} \Cati$ to the morphism $\Cati \xra{(-)^{op}} \Cati$).  By the naturality of the Grothendieck construction, it follows that we have a pullback square
\[ \begin{tikzcd}
\opLax(\C) \arrow{r}{\sim} \arrow{d}[swap]{s_{\opLax(\C)}} & \Lax(\C^{op}) \arrow{d}{s_{\Lax(\C^{op})}} \\
\Cati \arrow{r}{\sim}[swap]{(-)^{op}} & \Cati
\end{tikzcd} \]
in $\Cati$.
\end{rem}

\begin{rem}
As an alternative to the construction of \cref{define op/lax overcat}, both $\Lax(\C)$ and $\opLax(\C)$ are simultaneously encoded in the ``(op)lax square'' of $\Cati$ as constructed in \cite[\sec 5]{JFS-op/lax}.  While this alternative construction is both clean and aesthetically pleasing, we have chosen our own exposition in pursuit of the meta-goal of this paper (namely, to connect as many different concepts as possible to the Grothendieck construction).
\end{rem}

\subsection{The global co/limit functor}\label{subsection global co/limit}

Suppose we are given an arbitrary cocomplete $\infty$-category $\C$.  Then, the operation of taking colimits in $\C$ should be functorial in two different senses.
\begin{itemize}
\item On the one hand, colimits are functorial for natural transformations.  For instance, a natural transformation
\[ \begin{tikzcd}
\D \arrow[bend left=50]{r}{F}[swap, transform canvas={yshift=-0.8em}]{\Downarrow} \arrow[bend right=50]{r}[swap]{G} & \C
\end{tikzcd} \]
induces a canonical map $\colim_\D(F) \ra \colim_\D(G)$ in $\C$: a colimiting cocone $\D^\rightcone \ra \C$ extending $G$ composes to give an arbitrary cocone $\D^\rightcone \ra \C$ extending $F$, and then the desired map arises from the definition of a colimit as an initial cocone extending the given diagram.  Thus, taking colimits should give rise to a functor
\[ \Fun(\D,\C) \xra{\colim} \C . \]
\item On the other hand, colimits are also functorial for commutative diagrams of $\infty$-categories over $\C$.  For instance, a commutative diagram
\[ \begin{tikzcd}
\D \arrow{rr} \arrow{rd}[swap]{F} & & \E \arrow{ld}{G} \\
& \C
\end{tikzcd} \]
in $\Cati$ induces a canonical map $\colim_\D(F) \ra \colim_\E(G)$ in $\C$: a colimiting cocone $\E^\rightcone \ra \C$ extending $G$ composes to give an arbitrary cocone $\D^\rightcone \ra \E^\rightcone \ra \C$ extending $F$, and then the desired map once again arises from the definition of a colimit as an initial cocone extending the given diagram.  Thus, taking colimits should also give rise to a functor
\[ (\Cati)_{/\C} \xra{\colim} \C . \]
\end{itemize}
In fact, we have already seen a construction in \cref{define op/lax overcat} which unifies these two situations: via the cartesian fibration $\Lax(\C) \xra{s_{\Lax(\C)}} \Cati$, an arbitrary morphism
\[ \begin{tikzcd}
\D \arrow{rr}{H}[swap, transform canvas={yshift=-0.8em}]{\rotatebox[origin=c]{20}{$\Rightarrow$}} \arrow{rd}[swap]{F} & & \E \arrow{ld}{G} \\
& \C
\end{tikzcd} \]
in $\Lax(\C)$ gives rise to a unique diagram
\[ \begin{tikzcd}[row sep=1.5cm, column sep=1.5cm]
\D \arrow{rr}{H} \arrow[bend left]{rd}[pos=0.6]{G \circ H} \arrow[bend right]{rd}[sloped, pos=0.55, transform canvas={yshift=1.2em}]{\Uparrow}[swap]{F} & & \E \arrow{dl}{G} \\
& \C
\end{tikzcd} \]
in which the inner (commutative) triangle determines a cartesian arrow and the indicated natural transformation is a fiber morphism (lying over the object $\D \in \Cati$).  Thus, one might expect that it is possible to unify the above two senses in which colimits are functorial by means of a single functor
\[ \Lax(\C) \xra{\colim} \C . \]

The purpose of this subsection is to construct precisely such a \textit{global colimit functor}.  In fact, we will achieve this (as \cref{global colimit}) for an arbitrary $\infty$-category $\C$ (i.e.\! when $\C$ is not necessarily cocomplete), although in this more general setting we will of course need to restrict to the full subcategory of $\Lax(\C)$ spanned by those diagrams in $\C$ which admit a colimit (see \cref{notn subcat of lax with colim}).

\begin{rem}
A similar analysis suggests that when $\C$ is complete, there should exist a corresponding \textit{global limit functor} running
\[ \opLax(\C)^{op} \xra{\lim} \C . \]
Indeed, this can be obtained from \cref{global colimit} simply by taking opposites: if $\C$ is complete then $\C^{op}$ is cocomplete, and combining the resulting global colimit functor
\[ \Lax(\C^{op}) \xra{\colim} \C^{op} \]
with the equivalence of \cref{oplax of C is lax of Cop} and taking opposites yields a composite
\[ \opLax(\C)^{op} \xra{\sim} \Lax(\C^{op})^{op} \ra \C \]
which one can easily verify is the desired global limit functor.  (And of course, this equally well generalizes to the case that $\C$ only admits certain limits.)  We will therefore henceforth focus our attention only on the global \textit{colimit} functor.
\end{rem}

The first step in constructing the global colimit functor is to make the following reidentification of the nerve of the $\infty$-category $\Cati$ of $\infty$-categories.

\begin{lem}\label{nerve of Cati}
There is a canonical identification of $\Nervei(\Cati)_\bullet \in \CSS \subset s\S$ with the composite
\[ \bD^{op} \hookra (\Cati)^{op} \xra{\coCartFib(-)^{\simeq}} \S . \]
\end{lem}

\begin{proof}
For any $[n] \in \bD$, the Grothendieck construction provides an equivalence
\[ \Fun([n],\Cati) \xra[\sim]{\Gr} \coCartFib([n]) ; \]
passing to maximal subgroupoids, we obtain an equivalence
\[ \Nervei(\Cati)_n = \hom_\Cati([n],\Cati) \simeq \coCartFib([n])^\simeq . \]
The claim then follows from the naturality of the Grothendieck construction.
\end{proof}

Next, we would like to correspondingly identify the nerve of the lax overcategory of $\C$ (along with that of its source projection).  For this, observe (recalling \cref{ex lax from 1}) that the datum of a morphism in $\Lax(\C)$ is specified by the pair of
\begin{itemizesmall}
\item its image $[1] \xra{H} \Cati$ under the source projection, which by \cref{nerve of Cati} is equivalent to specifying a point
\[ \left( \Gr(H) \xra{\pr_{\Gr(H)}} [1] \right) \in \coCartFib([1])^\simeq \in \S , \]
along with
\item a map $\Gr(H) \ra \C$ from (the underlying $\infty$-category of) the ``directed mapping cylinder'' $\Gr(H)$ into our fixed target $\infty$-category $\C$.
\end{itemizesmall}
Moreover, a similar observation holds when we replace the object $[1] \in \bD$ by an arbitrary object $[n] \in \bD$.  Altogether, this identifies the canonical maps
\[ \Nervei(\Lax(\C))_n \ra \Nervei(\Cati)_n \]
(obtained by applying the functor
\[ \Nervei(-)_n \simeq \hom_\Cati([n],-) \in \Fun( \Cati , \S) \]
to the source projection map $\Lax(\C) \xra{s_{\Lax(\C)}} \Cati$) as being the cocartesian fibration associated to a certain map
\[ \Nervei(\Cati)_n \ra \S . \]
This motivates our desired identification (\cref{nerve of Lax}), but in order to state it precisely we first introduce the following notation.

\begin{notn}\label{notation UC and UC-dagger}
To ease notation, for any $\infty$-category $\C$ we denote by $\forget_\C$ any sub-composite of the composite
\[ \coCartFib(\C)^\simeq \hookra \coCartFib(\C)  \xra{\forget_{\coCartFib(\C)}} (\Cati)_{/\C} \ra \Cati \]
of the inclusion of the maximal subgroupoid with the two evident forgetful functors.  Moreover, we denote by $\forget_\C^\dagger$ any sub-composite beginning at $\coCartFib(\C)^\simeq$ of the composite
\[ \coCartFib(\C)^\simeq \xra{\sim} \left( \coCartFib(\C)^\simeq \right)^{op} \hookra \left( \coCartFib(\C) \right)^{op}  \xra{\forget_{\coCartFib(\C)}^{op}} \left( (\Cati)_{/\C} \right)^{op} \ra \left( \Cati \right)^{op} \]
of the canonical equivalence followed by the opposite of the above composite.  We also use this same notation when restricting to the subcategory $\LFib(\C) \subset \coCartFib(\C)$ or to its maximal subgroupoid.
\end{notn}

\begin{lem}\label{nerve of Lax}
Fix any $\C \in \Cati$.
\begin{enumerate}
\item\label{level of nerve of Lax}
For any $n \geq 0$, there is a canonical equivalence
\[ \Nervei(\Lax(\C))_n \simeq \forget_{\coCartFib([n])^\simeq} \left( \Gr \left( \coCartFib([n])^\simeq \xra{\forget_{[n]}^\dagger(-)} (\Cati)^{op} \xra{\hom_\Cati(-,\C)} \S \right) \right) \]
in $\S$.
\item\label{all of nerve of Lax}
The equivalences of part \ref{level of nerve of Lax} assemble into a canonical identification of $\Nervei(\Lax(\C))_\bullet \in \CSS \subset s\S$ with the composite
\[ \bD^{op} \hookra (\Cati)^{op} \xra{\forget_{\coCartFib(-)^\simeq} \left( \Gr \left(
\coCartFib(-)^\simeq \xra{\forget_{(-)}^\dagger(=)} (\Cati)^{op} \xra{\hom_\Cati(=,\C)} \S
\right) \right)} \S . \]
\end{enumerate}
\end{lem}

\begin{proof}
Part \ref{level of nerve of Lax} follows directly from \cref{define op/lax overcat}, and part \ref{all of nerve of Lax} follows from the naturality of the Grothendieck construction.
\end{proof}

\begin{notn}\label{notn subcat of lax with colim}
We denote by $\Lax(\C)^{\colim} \subset \Lax(\C)$ the full subcategory on those functors $\D \xra{F} \C$ which admit a colimit in $\C$.
\end{notn}

We can now give the main result of this section, whose output we refer to as the \bit{global colimit functor} for $\C$.


\begin{prop}\label{global colimit}
For any $\C \in \Cati$, there is a functor
\[ \Lax(\C)^{\colim} \xra{\colim} \C \]
which takes an object $(\D \xra{F} \C) \in \Lax(\C)^{\colim}$ to $\colim_\D(F) \in \C$.  Moreover, this functor
\begin{itemize}
\item restricts to the usual colimit functor on each fiber
\[ \Fun(\D,\C)^{\colim} = \Fun(\D,\C) \cap \Lax(\C)^{\colim} , \]
and
\item takes a commutative diagram
\[ \begin{tikzcd}
\D \arrow{rr} \arrow{rd}[swap]{F} & & \E \arrow{ld}{G} \\
& \C
\end{tikzcd} \]
(considered as a cartesian morphism in $\Lax(\C)^{\colim}$) to the canonical induced map
\[ \colim_\D F \ra \colim_\E G \]
in $\C$.
\end{itemize}
\end{prop}

\begin{proof}
In order to prove the claim, we make the following construction.  Define $\Lax(\C)' \in \Cati$ to be the unique $\infty$-category such that $\Nervei(\Lax(\C)')_\bullet \in \CSS \subset s\S$ is given by
\[ \bD^{op} \hookra (\Cati)^{op} \xra{\forget_{\coCartFib(-)^\simeq} \left( \Gr \left(
\coCartFib(-)^\simeq \xra{\forget_{(-)}^\dagger \left( = \underset{(-)}{\diamond} (-) \right)} (\Cati)^{op} \xra{\hom_\Cati(=,\C)} \S
\right) \right)_!} \S , \]
where the subscript decorating the second functor indicates that we are restricting to the subspace corresponding to those pairs of a functor $[n] \xra{F} \Cati$ and a map
\[ \forget_{[n]}^\dagger \left( \Gr(F) \underset{[n]}{\diamond} [n] \right) \ra \C \]
such that the induced composite diagram
\[ \begin{tikzcd}
\forget_{[n]}^\dagger ( \Gr(F) ) \arrow{r} \arrow[hook]{d} & \C \\
\forget_{[n]}^\dagger \left( \Gr(F) \underset{[n]}{\diamond} [n] \right) \arrow{ru}
\end{tikzcd} \]
defines a left Kan extension (along the inclusion of a full subcategory).  The canonical inclusions
\[ \Gr(F) \hookra \Gr(F) \underset{[n]}{\diamond} [n] \hookla [n] \]
induce maps
\[ \Nervei(\Lax(\C))_\bullet \la \Nervei(\Lax(\C)')_\bullet \ra \Nervei(\C)_\bullet \times \Nervei(\Cati)_\bullet \]
in $\CSS \subset s\S$, where the identification of $\Nervei(\Lax(\C))_\bullet$ comes from \cref{nerve of Lax}\ref{all of nerve of Lax} and the identification of $\Nervei(\C)_\bullet \times \Nervei(\Cati)_\bullet$ follows from \cref{nerve of Cati} and \cref{Gr of a constant functor}.  It then follows from Proposition T.4.3.2.15 that we have an induced factorization
\[ \begin{tikzcd}
\Lax(\C) & \Lax(\C)' \arrow{l} \arrow[dashed]{ld}[sloped]{\sim} \\
\Lax(\C)^{\colim} \arrow[hook]{u}
\end{tikzcd} \]
via an equivalence in $\Cati$.  Moreover, combining this diagram with the composite $\Lax(\C)' \ra \C \times \Cati \ra \C$ with the projection map gives us a unique functor
\[ \begin{tikzcd}
\Lax(\C) & \Lax(\C)' \arrow{l} \arrow{r} \arrow{ld}[sloped]{\sim} & \C \\
\Lax(\C)^{\colim} \arrow[hook]{u} \arrow[dashed, bend right=15]{rru}
\end{tikzcd} \]
making the diagram commute.  By Proposition T.4.3.3.10, this is precisely the desired functor; moreover, it is clear from the construction that it encodes the asserted functorialities.
\end{proof}

\begin{rem}\label{functoriality of global colim}
Clearly, the global colimit functor (\cref{global colimit}) is itself functorial in the following sense: if $\C_1 \ra \C_2$ is a functor which commutes with all colimits existing in $\C_1$, then we obtain a commutative square
\[ \begin{tikzcd}[column sep=1.5cm]
\Lax(\C_1)^{\colim} \arrow{r}{\colim^{\C_1}} \arrow{d} & \C_1 \arrow{d} \\
\Lax(\C_2)^{\colim} \arrow{r}[swap]{\colim^{\C_2}} & \C_2
\end{tikzcd} \]
in $\Cati$.
\end{rem}

\begin{rem}
One could also construct the global colimit functor (i.e.\! prove \cref{global colimit}) in the following way.  First of all, the Grothendieck construction of the source projection
\[ \Lax(\C) \xra{s_{\Lax(\C)}} \Cati \]
produces a ``tautological bundle'' over $\Lax(\C)$, a cocartesian fibration whose fiber over an object $(\D \xra{F} \C) \in \Lax(\C)$ is $\D$ itself.  This can moreover be shown to admit a tautological map to $\C$ which, restricted to such a fiber, is precisely the functor $\D \xra{F} \C$.  The global colimit functor can then be produced by appealing to Proposition T.4.2.2.7.  However, this method is in fact quite a bit more involved than the route we have taken here.
\end{rem}

\section{Homotopy pullbacks in $(\Cati)_\Thomason$, finality, and Theorems {\thmA}, {\Bn}, and {\Cn}}\label{section ho-p.b.'s in Thomason}

Via the Thomason model structure on $\Cati$ of \cref{section thomason}, we can consider $\infty$-categories as ``presentations of spaces''; the corresponding localization functor $\Cati \ra \loc{\Cati}{\bW_\Thomason} \simeq \S$ is that of groupoid completion.  Being a left adjoint, this functor commutes with colimits, but in general its interplay with limits is much more complicated.  In this section, we describe certain sufficient conditions under which it commutes with a given pullback.

In the 1-categorical case, there is a long history of results of this variety, going back to Quillen's celebrated \cite[Theorem B]{QuillenAKTI}.  The current state of the art seems to be Barwick--Kan's pair of results \cite[Theorems {\Bn} (5.6) and {\Cn} (5.8)]{BK-partial} (the former generalizing Dwyer--Kan--Smith's \cite[Theorem {\Bn} (6.2)]{DKS-hocommdiag}, the latter identical to their \cite[Theorem {\Cn} (6.4)]{DKS-hocommdiag}), as described in \cref{subsection outline}.

The main goal of this section is to give $\infty$-categorical generalizations of these results; these appear in \cref{subsection thms Bn and Cn}.  In \cref{subsection property Q etc} we work towards this goal with a pair of foundational results surrounding homotopy pullbacks in $(\Cati)_\Thomason$ (whose 1-categorical analogs constitute the main input to the proof of \cite[Theorem B]{QuillenAKTI}), and in \cref{subsection thm A} we take a moment to briefly restate Joyal's quasicategorical analog (namely Theorem T.4.1.3.1) of Quillen's \cite[Theorem A]{QuillenAKTI} in invariant langauge (i.e.\! stated in $\Cati$ instead of in $s\Set_\Joyal$).

\subsection{Homotopy pullbacks in $(\Cati)_\Thomason$: a first pass}\label{subsection property Q etc}

In and of themselves, pullbacks among $\infty$-categories are relatively understandable; for instance, limits commute with the right adjoint of the composite adjunction
\[ \begin{tikzcd}[column sep=1.5cm]
s\S 
{\arrow[transform canvas={yshift=0.7ex}]{r}{\leftloc_\CSS}[swap, transform canvas={yshift=0.25ex}]{\scriptstyle \bot} \arrow[transform canvas={yshift=-0.7ex}, hookleftarrow]{r}[swap]{\forget_\CSS}}
& \CSS
{\arrow[transform canvas={yshift=0.7ex}]{r}{\Nervei^{-1}}[swap, transform canvas={yshift=0.05ex}]{\sim} \arrow[transform canvas={yshift=-0.7ex}, leftarrow]{r}[swap]{\Nervei}}
& \Cati .
\end{tikzcd} \]
On the other hand, it is a subtle question to determine when such a pullback commutes with groupoid completion (or, working in complete Segal spaces, with geometric realization (recall \cref{rnerves:groupoid-completion of CSSs})).  In this subsection, we address this question in the special case that one of the maps in the pullback is a special sort of cocartesian fibration.

We begin with the relevant definition, an analog of \cite[4.4]{QuillenAKTI}.

\begin{defn}\label{define prop Q}
We say that a functor $\C \xra{F} \Cati$ has \bit{property {\propQ}} if it factors through $\bW_\Thomason \subset \Cati$.
\end{defn}

The importance of \cref{define prop Q} stems from the following result, which we call \bit{Lemma {\propQ}}, an analog of \cite[Lemma 4.5]{BK-partial} (there called ``Quillen's lemma'').

\begin{lem}\label{quillen's lemma}
If $\C \xra{F} \Cati$ factors through $\bW_\Thomason \subset \Cati$, then for any $x \in \C$, the fiber inclusion
\[ \begin{tikzcd}
F(x) \arrow{r} \arrow{d} & \Gr(F) \arrow{d} \\
\{ x \} \arrow[hook]{r} & \C
\end{tikzcd} \]
is a homotopy pullback square in $(\Cati)_\Thomason$, i.e.\! it gives rise to a pullback square
\[ \begin{tikzcd}
F(x)^\gpd \arrow{r} \arrow{d} & \Gr(F)^\gpd \arrow{d} \\
\{ x \}^\gpd \arrow[hook]{r} & \C^\gpd
\end{tikzcd} \]
in $\S$.
\end{lem}

\begin{proof}
Let $\ttC \in s\Set_\Joyal^f$ be a quasicategory presenting $\C \in \Cati$, let $\ttD \ra \ttC$ be a JL-cocartesian fibration presenting $(\Gr(F) \ra \C) \in \coCartFib(\C)$, let $\ttD' \ra \ttC$ be a JL-left fibration presenting $(\leftloc_{\LFib(\C)}(\Gr(F)) \ra \C) \in \LFib(\C)$, and let $\ttx \in \ttC_0$ be a vertex corresponding to $x \in \C$ (see \cite[\cref{grjl:define JL-stuff}]{grjl}).  Let us define $\ttF,\ttF' \in s\Set$ via the pullback squares
\[ \begin{tikzcd}
\ttF \arrow{r} \arrow{d} & \ttD \arrow{d} \\
\pt_{s\Set} \arrow{r}[swap]{\ttx} & \ttC
\end{tikzcd} \]
and
\[ \begin{tikzcd}
\ttF' \arrow{r} \arrow{d} & \ttD' \arrow{d} \\
\pt_{s\Set} \arrow{r}[swap]{\ttx} & \ttC
\end{tikzcd} \]
in $s\Set$.  Considered in $s\Set_\Joyal$, these present fiber inclusions, the first of which is
\[ \begin{tikzcd}
F(x) \arrow{r} \arrow{d} & \Gr(F) \arrow{d} \\
\{ x \} \arrow[hook]{r} & \C
\end{tikzcd} \]
and the second of which, by \cref{cocartfib to lfib corresponds to groupoid-completion}, we can identify as
\[ \begin{tikzcd}
F(x)^\gpd \arrow{r} \arrow{d} & \leftloc_{\LFib(\C)}(\Gr(F)) \arrow{d} \\
\{ x \} \arrow[hook]{r} & \C .
\end{tikzcd} \]
Moreover, by Proposition T.2.1.3.1, the assertion that
\[ \C \xra{F} \Cati \xra{(-)^\gpd} \S \]
factors through $\S^\simeq \subset \S$ is equivalent to the assertion that the map $\ttD' \ra \ttC$ is in fact in $\bF_\KQ$.  Since $s\Set_\KQ$ is right proper (see \cite[Theorem 13.1.13]{Hirsch}), it follows that the pullback square
\[ \begin{tikzcd}
\ttF' \arrow{r} \arrow{d} & \ttD' \arrow[two heads]{d} \\
\pt_{s\Set} \arrow{r}[swap]{\ttx} & \ttC
\end{tikzcd} \]
in $s\Set_\KQ$ is also a homotopy pullback square, which implies that
\[ \begin{tikzcd}
F(x)^\gpd \arrow{r} \arrow{d} & \leftloc_\Fib(\Gr(F))^\gpd \arrow{d} \\
\{ x \}^\gpd \arrow[hook]{r} & \C^\gpd
\end{tikzcd} \]
is a pullback square in $\S$.  By \cref{always groupoid-complete} we obtain an equivalence
\[ \Gr(F)^\gpd \xra{\sim} (\leftloc_{\LFib(\C)}(\Gr(F)))^\gpd \]
in $\S_{/\C^\gpd}$, which completes the proof.
\end{proof}

We also have the following parametrized version of Lemma {\propQ} (\ref{quillen's lemma}).

\begin{cor}\label{pullback of cocart fibn}
If $\C \xra{F} \Cati$ has property {\propQ}, then for any $\D \xra{G} \C$, the resulting pullback $\D \times_\C \Gr(F)$ is a homotopy pullback in $(\Cati)_\Thomason$.
\end{cor}

\begin{proof}
By the naturality of the Grothendieck construction, we have a commutative square
\[ \begin{tikzcd}
\Gr(F \circ G) \arrow{r} \arrow{d} & \Gr(F) \arrow{d} \\
\D \arrow{r}[swap]{G} & \C
\end{tikzcd} \]
in $\Cati$, where both of the vertical maps are cocartesian fibrations.  We would like to show that this becomes a pullback square upon application of $(-)^\gpd : \Cati \ra \S$.  For this, note that any map $\pt_\S \xra{x} \D^\gpd$ in $\S$ comes from a map $\pt_{\Cati} \xra{\tilde{x}} \D$ in $\Cati$.  Then, we have the diagram
\[ \begin{tikzcd}
& \Gr(F \circ G) \arrow{rr} \arrow{dd} & & \Gr(F) \arrow{dd} \\
(F\circ G)(x) \arrow[crossing over]{rr}[pos=0.6]{\sim} \arrow{ru} \arrow{dd} & & F(G(x)) \arrow{ru} \\
& \D \arrow{rr}[pos=0.4]{G} & & \C \\
\{ \tilde{x} \} \arrow{rr}[swap, pos=0.6]{\sim} \arrow[hook]{ru} & & \{ G(\tilde{x}) \} \arrow[leftarrow, crossing over]{uu} \arrow[hook]{ru}
\end{tikzcd} \]
in $\Cati$, whose back face is the above commutative square and in which the upper oblique arrows are the fiber inclusions over $\tilde{x} \in \D$ and $G(\tilde{x}) \in \C$.  By Lemma {\propQ} (\ref{quillen's lemma}), we obtain that applying $(-)^\gpd : \Cati \ra \S$ to this commutative diagram yields a commutative diagram in $\S$ in which the oblique squares are pullbacks.  Hence, the square
\[ \begin{tikzcd}
\Gr(F \circ G)^\gpd \arrow{r} \arrow{d} & \Gr(F)^\gpd \arrow{d} \\
\D^\gpd \arrow{r}[swap]{G^\gpd} & \C^\gpd
\end{tikzcd} \]
is a pullback in $\S$, since for every point of $\D^\gpd$ the induced map on corresponding fibers is an equivalence.
\end{proof}

\subsection{Finality and Theorem A}\label{subsection thm A}

In this subsection, we briefly recall a few definitions and results from \sec T.4.1, restating them in invariant language.

\begin{defn}\label{define final}
A functor $\I \xra{F} \J$ is called \bit{final} if, for any functor $\J \xra{G} \C$ such that $\colim_\J G$ exists, the colimit $\colim_\I (G \circ F)$ also exists and the natural map
\[ \colim_\I (G \circ F) \ra \colim_\J G \]
(in the sense of the global colimit functor (\cref{global colimit})) is an equivalence in $\C$.

Dually, a functor $\I \xra{F} \J$ is called \bit{initial} if, for any functor $\J \xra{G} \C$ such that $\lim_\J G$ exists, the limit $\lim_\I(G \circ F)$ also exists and the natural map
\[ \lim_\J(G) \ra \lim_\I(G \circ F) \]
(in the sense of the global limit functor (the dual of \cref{global colimit})) is an equivalence in $\C$.
\end{defn}

\begin{rem}
A functor $\I \ra \J$ is initial if and only if its opposite $\I^{op} \ra \J^{op}$ is final.
\end{rem}

\begin{rem}
The notion of finality given in \cref{define final} is also sometimes called ``cofinality'' or ``right cofinality'', while that initiality is also sometimes called ``co-cofinality'' or ``left cofinality''.  We have chosen our terminology because it seems most natural: the simplest example of a final functor is the inclusion $\{ \pt_\I \} \hookra \I$ of a final object, while the simplest example of an initial functor is the inclusion $\{ \es_\I \} \hookra \I$ of an initial object.
\end{rem}

\begin{rem}\label{rem final is cofinal}
By Proposition T.4.1.1.8 (see also Corollary T.4.1.1.10), the notion of finality given in \cref{define final} is an invariant version of the quasicategorical notion of ``cofinality'' (given e.g.\! as Definition T.4.1.1.1).
\end{rem}

\begin{prop}\label{final functor is an equivalence on groupoid-completions}
Every final functor lies in $\bW_\Thomason \subset \Cati$.
\end{prop}

\begin{proof}
This follows from Proposition T.4.1.1.3(3).
\end{proof}

\begin{prop}\label{product of final functors is final}
If $\I_1 \xra{F_1} \J_1$ and $\I_2 \xra{F_2} \J_2$ are both final, then so is $\I_1 \times \I_2 \xra{F_1 \times F_2} \J_1 \times \J_2$.
\end{prop}

\begin{proof}
By Corollary T.4.1.1.13, both of the functors in the composite
\[ \I_1 \times \I_2 \xra{F_1 \times \id_{\I_2}} \J_1 \times \I_2 \xra{\id_{\J_1} \times F_2} \J_1 \times \J_2 \]
are final, and hence by Proposition T.4.1.1.3(2) their composite $F_1 \times F_2$ is also final.\footnote{This is of course closely related to Fubini's theorem for colimits, but to deduce it from that result would require slightly trickier manipulations depending on which of the various colimits actually exist.}
\end{proof}

The next result is the main point of this subsection, Joyal's $\infty$-categorical analog of \cite[Theorem A]{QuillenAKTI}; we refer to it simply as \bit{Theorem {\thmA}}.

\begin{thm}\label{theorem A}
A functor $\I \xra{F} \J$ is final iff for every object $j \in \J$,
\[ (\I \times_\J \J_{j/})^\gpd \simeq \pt_\S . \]
\end{thm}

\begin{proof}
This follows from Theorem T.4.1.3.1; by the Reedy trick (and the implications of Remark T.2.0.0.5 and Corollary T.2.1.2.2), the pullback given there is a homotopy pullback in $s\Set_\Joyal$.
\end{proof}

\begin{cor}\label{groupoid-completion of infty-cat with terminal object is contractible}
If $\J \in \Cati$ has a terminal object, then $\J^\gpd \simeq \pt_\S$.
\end{cor}

\begin{proof}
We apply \cref{theorem A} to deduce that the functor $\{ \pt_\J \} \hookra \J$ is final: for any $j \in \J$, we have that $\{ \pt_\J \} \times_\J \J_{j/} \simeq \hom_\J(j,\pt_\J) \simeq \pt_\S$.  Hence, the claim follows from \cref{final functor is an equivalence on groupoid-completions}.
\end{proof}

\begin{rem}
One could also prove \cref{groupoid-completion of infty-cat with terminal object is contractible} by observing that a functor $\pt_\Cati \ra \J$ is a right adjoint if (and only if) it selects a terminal object, and then appealing to \cref{rnerves:adjns induce equivces on gpd-complns}.
\end{rem}

\subsection{Theorems {\Bn} and {\Cn}: a second pass at homotopy pullbacks in $(\Cati)_\Thomason$}\label{subsection thms Bn and Cn}

In this final subsection, we provide $\infty$-categorical generalizations of Barwick--Kan's pair of results \cite[Theorems {\Bn} (5.6) and {\Cn} (5.8)]{BK-partial}.

\begin{rem}
The results of this subsection will be used in \cref{fundthm:section equivalence of 7 with hom in loc}.  In fact, there will actually employ their dual formulations.  Our choice of variances in this subsection are so that our exposition adheres as closely as possible to that of \cite[\sec 5]{BK-partial}.
\end{rem}

We begin with the following.

\begin{notn}\label{define walking zigzag categories}
For $n \geq 1$, we define $\word{z}_n \in \strcat$ by the pattern
\begin{align*}
\word{z}_1 & = \left( s \ra t \right) , \\
\word{z}_2 &= \left( s \la \bullet \ra t \right) , \\
\word{z}_3 &= \left( s \ra \bullet \la \bullet \ra t \right) , \\
\word{z}_4 &= \left( s \la \bullet \ra \bullet \la \bullet \ra t \right) ,
\end{align*}
etc.\! (where we have named only the leftmost and rightmost objects of these categories).
\end{notn}

We now give the following omnibus definition.

\begin{defn}\label{define phpb}
For any $n \geq 1$ and any $\C \in \Cati$, we define $(\C \da_n \C) = \Fun(\word{z}_n,\C)$.  Evaluation at the objects $s ,t \in \word{z}_n$ induces maps
\[ \C \xla{s} (\C \da_n \C) \xra{t} \C . \]
More generally, for any functors $\D \xra{F} \C$ and $\E \xra{G} \C$ and any objects $d \in \D$ and $e \in \E$, we define new $\infty$-categories and maps between them via the induced diagram
\[ \begin{tikzcd}
(F(d) \da_n G(e)) \arrow{r} \arrow{d} & (F(\D) \da_n G(e)) \arrow{r} \arrow{d} & (\C \da_n G(e)) \arrow{r} \arrow{d} & \pt_\Cati \arrow{d}{e} \\
(F(d) \da_n G(\E)) \arrow{r} \arrow{d} & (F(\D) \da_n G(\E)) \arrow{r} \arrow{d} & (\C \da_n G(\E)) \arrow{r} \arrow{d} & \E \arrow{d}{G} \\
(F(d) \da_n \C) \arrow{r} \arrow{d} & (F(\D) \da_n \C) \arrow{r} \arrow{d} & (\C \da_n \C) \arrow{r}[swap]{t} \arrow{d}{s} & \C \\
\pt_\Cati \arrow{r}[swap]{d} & \D \arrow{r}[swap]{F} & \C
\end{tikzcd} \]
in $\Cati$ in which all squares are pullbacks.  Thus, we may consider $(F(\D) \da_n G(\E))$ as simultaneously generalizing all three constructions $(F(\D) \da_n \C)$, $(\C \da_n G(\E))$ and $(\C \da_n \C)$, with the convention that if either $F$ or $G$ is simply $\id_\C$ then we omit it from the notation; we refer to any of these constructions (but especially to $(F(\D) \da_n G(\E))$) as a \bit{potential homotopy pullback $\infty$-category}.  Similarly, the construction $(F(d) \da_n G(\E))$ generalizes the construction $(F(d) \da_n \C)$ (in the case that $G = \id_\C$), while the construction $(F(\D) \da_n G(e))$ generalizes the construction $(\C \da_n G(e))$ (in the case that $F = \id_\C$).  Additionally, we denote all vertical maps landing at $\D$ (and in particular, all vertical maps in the third column landing at $\C$) by $s$ and refer to these as \bit{source} maps, and we denote all horizontal maps landing at $\E$ (and in particular, all horizontal maps in the third row landing at $\C$) by $t$ and refer to these as \bit{target} maps.
\end{defn}

\begin{rem}
To make sense of the terminology, one should think of the construction $(F(\D) \da_n G(\E))$ of \cref{define phpb} as a sort of ``directed'' analog of the standard explicit construction of a \textit{homotopy pullback} of topological spaces (as the space of pairs of points in the two sources of the cospan equipped with a path between their images in the common target).  The question, of course, is whether this actually computes the homotopy pullback in $(\Cati)_\Thomason$ (which explains the word ``potential'' in the name); a sufficient condition for this to be the case is precisely the content of Theorem {\Bn} (\ref{thm Bn}).  Continuing along these lines, one might think of $(F(\D) \da_n \C)$ and $(\C \da_n G(\E))$ as ``directed'' analogs of the \textit{mapping path space} construction, i.e.\! the standard explicit factorization of an arbitrary map of topological spaces as a weak equivalence followed by a fibration.  The reader may find these analogies helpful to keep in mind while reading the rest of this subsection.
\end{rem}

In order to simultaneously deal with the case when $n$ is even and when $n$ is odd, we also introduce the following.

\begin{notn}\label{half-open interval notn}
Inspired by the half-open intervals $\eveninterval$  and $\oddinterval$, we will enclose an expression by $\evendash$ when we mean for it to be read only when $n$ is even, while when we will enclose an expression by $\odddash$ when we mean for it to be read only when $n$ is odd.  So for instance, $\C^\evenop$ denotes $\C^{op}$ when $n$ is even, and simply denotes $\C$ when $n$ is odd.
\end{notn}

\begin{lem}\label{phpb gives fibns}
Let $n \geq 1$, let $\D \xra{F} \C$ and $\E \xra{G} \C$ be any functors.
\begin{enumerate}

\item\label{full phpb is fibn}

We have
\[ \left( (F(\D) \da_n G(\E)) \xra{s} \D \right) \in \evencoCartFib(\D) \]
and
\[ \left( (F(\D) \da_n G(\E)) \xra{t} \E \right) \in \coCartFib(\E) . \]

\item\label{fiber of phpb is fibn}

For any objects $d \in \D$ and $e \in \E$, we have
\[ \left( (F(\D) \da_n G(e)) \xra{s} \D \right) \in \evencoCartFib(\D) \]
and 
\[ \left( (F(d) \da_n G(\E)) \xra{t} \E \right) \in \coCartFib(\E) . \]
\end{enumerate}
\end{lem}

\begin{proof}
In both parts, we will only prove the second of the two claims; the first claims follow from nearly identical arguments.  Moreover, since cocartesian fibrations are stable under pullback, it suffices to prove these statements in the case that $G = \id_\C$.  For this, let us denote by $t' \in \word{z}_n$ the penultimate object (reading from left to right), and let us denote by $\word{z}_n' \subset \word{z}_n$ the full subcategory on all the objects besides $t \in \word{z}_n$, so that we can identify $\word{z}_n$ as a pushout
\[ \word{z}_n \simeq \word{z}_n' \coprod_{t' , \pt_\Cati , 0 } [1] \]
in $\Cati$.  We therefore obtain a diagram
\[ \begin{tikzcd}
(F(d) \da_n \C) \arrow{r} \arrow{d} & (F(\D) \da_n \C) \arrow{r} \arrow{d} & (\C \da_n \C) \arrow{r} \arrow{d} & \Fun([1],\C) \arrow{r}{\ev_1} \arrow{d}{\ev_0} & \C \\
\pt_\Cati \underset{F(d) , \C , \ev_s}{\times} \Fun(\word{z}_n' , \C) \arrow{r} \arrow{d} & \D \underset{F , \C , \ev_s}{\times} \Fun(\word{z}_n' , \C) \arrow{r} \arrow{d} & \Fun(\word{z}_n' , \C) \arrow{r}[swap]{\ev_{t'}} \arrow{d}{\ev_s} & \C \\
\pt_\Cati \arrow{r}[swap]{d} & \D \arrow{r}[swap]{F} & \C
\end{tikzcd} \]
in $\Cati$ in which all squares are pullbacks and the upper row contains as composites the two functors $(F(\D) \da_n \C) \xra{t} \C$ and $(F(d) \da_n \C) \xra{t} \C$ which we would like to show are cocartesian fibrations.  Both claims now follow by applying the dual of Corollary T.2.4.7.12 to the corresponding composites
\[ \D \underset{F , \C , \ev_s}{\times} \Fun(\word{z}_n',\C) \ra \C \]
and 
\[ \pt_\Cati \underset{F(d) , \C , \ev_s}{\times} \Fun(\word{z}_n',\C) \ra \C \]
in the middle row.
\end{proof}

\begin{notn}\label{notn phpb as fctr to Cati}
We denote the functors classifying the co/cartesian fibrations of \cref{phpb gives fibns}\ref{full phpb is fibn} by
\[ \D^\oddop \xra{(F(-) \da_n G(\E))} \Cati \]
and
\[ \E \xra{(F(\D) \da_n G(-))} \Cati , \]
and we denote the functors classifying the co/cartesian fibrations of \cref{phpb gives fibns}\ref{fiber of phpb is fibn} by
\[ \D^\oddop \xra{(F(-) \da_n G(e))} \Cati \]
and
\[ \E \xra{(F(d) \da_n G(-))} \Cati \]
(again omitting the functor $F$ or $G$ if it is just $\id_\C$).  Note that this is indeed consistent with the notation given in \cref{define phpb}, and identifies all of the various pullbacks over $\pt_\Cati$ in the diagram given there as fiber inclusions.  (In particular, the $\infty$-category $(F(d) \da_n G(e))$ includes as a fiber of two different co/cartesian fibrations.)
\end{notn}

\begin{rem}\label{lift to fctrs to fibns over other}
We may consider \cref{phpb gives fibns}\ref{fiber of phpb is fibn} as equipping the first two functors defined in \cref{notn phpb as fctr to Cati} with lifts
\[ \begin{tikzcd}[column sep=2.5cm]
& \coCartFib(\E) \arrow{d} \\
\D^\oddop \arrow[dashed]{ru} \arrow{r}[swap]{(F(-) \da_n G(\E))} & \Cati
\end{tikzcd} \]
and
\[ \begin{tikzcd}[column sep=2.5cm]
& \evencoCartFib(\D) \arrow{d} \\
\E \arrow[dashed]{ru} \arrow{r}[swap]{(F(\D) \da_n G(-))} & \Cati
\end{tikzcd} \]
(through the evident forgetful functors).  (Preservation of co/cartesian morphisms follows from Corollary T.2.4.7.12.)
\end{rem}

We then have the following elementary but important observations.

\begin{prop}\label{common sections and htpy equivces}
Let $n \geq 1$, and choose any functor $\D \xra{F} \C$ and $\E \xra{G} \C$.
\begin{enumerate}
\item\label{common section} The unique functor $\word{z}_n \ra \pt_\Cati$ induces a common section
\[ \begin{tikzcd}
(\C \da_n \C) \arrow[transform canvas={yshift=+0.7ex}]{r}{s} \arrow[transform canvas={yshift=-0.7ex}]{r}[swap]{t} & \C \arrow[dashed, bend right, out=-40]{l}[swap]{q}
\end{tikzcd} \]
to the source and target maps.
\item\label{induce sections} The common section of part \ref{common section} induces sections
\[ \begin{tikzcd}
(F(\D) \da_n \C) \arrow{r}[swap]{s} & \D \arrow[dashed, bend right, out=-40]{l}[swap]{q}
\end{tikzcd} \]
and
\[ \begin{tikzcd}
(\C \da_n G(\E)) \arrow{r}[swap]{t} & \E . \arrow[dashed, bend right, out=-40]{l}[swap]{q}
\end{tikzcd} \]
\item\label{sections are htpy equivces} The section diagrams of part \ref{induce sections} (and in particular, those of part \ref{common section}) define homotopy equivalences in $(\Cati)_\Thomason$.
\end{enumerate}
\end{prop}

\begin{proof}
Part \ref{common section} follows from the fact that both composites $\pt_\Cati \rra \word{z}_n \ra \pt_\Cati$ are canonically equivalent to $\id_{\pt_\Cati}$.  Then, part \ref{induce sections} follows from part \ref{common section} and the definitions of $(F(\D) \da_n \C)$ and $(\C \da_n G(\E))$ as pullbacks.  For part \ref{sections are htpy equivces}, observe first that either composite $\word{z}_n \ra \pt_\Cati \rra \word{z}_n$ is connected by a zigzag of natural transformations to $\id_{\word{z}_n}$; moreover, working with $\pt_\Cati \xra{s} \word{z}_n$ this zigzag can be taken in $\strcat_*$ where we point our categories using their source objects (i.e.\! such that all the constituent natural transformations of the zigzag have the map $\id_s$ as their component at $s$), and similarly for working with $\pt_\Cati \xra{t} \word{z}_n$.  By applying \cref{hammocks:nat w.e. induces nat trans} (where we take $\C$ to be equipped with the maximal relative structure), we see that either composite $(\C \da_n \C) \rra \C \ra (\C \da_n \C)$ is in turn connected to $\id_{(\C \da_n \C)}$ by a zigzag of natural transformations, such that all of the constituent natural transformations commute with the chosen projection $(\C \da_n \C) \ra \C$ (either the source or target map).  Since the functor $\Cati \xra{- \times [1]} \Cati$ commutes with pullbacks (being a limit), by the functoriality of pullbacks these induce zigzags of natural transformations between the composite
\[ (F(\D) \da_n \C) \xra{s} \D \ra (F(\D) \da_n \C) \]
and $\id_{(F(\D) \da_n \C)}$ and between the composite
\[ (\C \da_n G(\E)) \xra{t} \E \ra (\C \da_n G(\E)) \]
and $\id_{(\C \da_n G(\E))}$.  Thus, the claim follows from \cref{rnerves:nat trans induces equivce betw maps on gpd-complns}.
\end{proof}

We now define the key concept of this subsection.

\begin{defn}
We say that a functor $\D \xra{F} \C$ has \bit{property {\Bnbit}} if the functor
\[ \C \xra{(F(\D) \da_n - )} \Cati \]
has property {\propQ}.
\end{defn}

We now give the main result of this subsection, which we refer to as \bit{Theorem {\Bnbit}} (for homotopy pullbacks (in $(\Cati)_\Thomason$)).

\begin{thm}\label{thm Bn}
If the functor $\D \xra{F} \C$ has property {\Bn}, then for any functor $\E \xra{G} \C$, the (not generally commutative) square
\[ \begin{tikzcd}
(F(\D) \da_n G(\E)) \arrow{r}{s} \arrow{d}[swap]{t} & \D \arrow{d}{F} \\
\E \arrow{r}[swap]{G} & \C
\end{tikzcd} \]
is a homotopy pullback square in $(\Cati)_\Thomason$, i.e.\! it induces a (commutative) pullback square
\[ \begin{tikzcd}
(F(\D) \da_n G(\E))^\gpd \arrow{r}{s^\gpd} \arrow{d}[swap]{t^\gpd} & \D^\gpd \arrow{d}{F^\gpd} \\
\E^\gpd \arrow{r}[swap]{G^\gpd} & \C^\gpd
\end{tikzcd} \]
in $\S$.
\end{thm}

\begin{proof}
To say that $\D \xra{F} \C$ has property {\Bn} is to say that the functor
\[ \C \xra{(F(\D) \da_n - )} \Cati \]
has property {\propQ}.  By the naturality of the Grothendieck construction, the functor
\[ \E \xra{(F(\D) \da_n G(-))} \Cati \]
is precisely the composite
\[ \E \xra{G} \C \xra{(F(\D) \da_n - )} \Cati , \]
and therefore also has property {\propQ}.  Hence, by Lemma {\propQ} (\ref{quillen's lemma}), for any objects $c \in \C$ and $e \in \E$ the fiber inclusions
\[ \begin{tikzcd}
(F(\D) \da_n c) \arrow{r} \arrow{d} & (F(\D) \da_n \C) \arrow{d}{t} \\
\{ c \} \arrow[hook]{r} & \C
\end{tikzcd} \]
and
\[ \begin{tikzcd}
(F(\D) \da_n G(e)) \arrow{r} \arrow{d} & (F(\D) \da_n G(\E)) \arrow{d}{t} \\
\{ e \} \arrow{r} & \E
\end{tikzcd} \]
are both homotopy pullback squares in $(\Cati)_\Thomason$.

Now, observe that we have a diagram
\[ \begin{tikzcd}
(F(\D) \da_n G(\E)) \arrow{r} \arrow[bend left]{rr}{s} \arrow{d}[swap]{t} & (F(\D) \da_n \C) \arrow{d}[swap]{t} \arrow{r}{s} & \D \arrow[bend left, out=50]{l}{q}[swap, transform canvas={yshift=+0.7ex}]{\approx} \arrow[bend left]{ld}{F} \\
\E \arrow{r}[swap]{G} & \C
\end{tikzcd} \]
in $(\Cati)_\Thomason$, in which
\begin{itemize}
\item the map labeled $q$ comes from \cref{common sections and htpy equivces}\ref{induce sections}, 
\item every bounded connected region is commutative except for the one containing the symbol $\approx$, which is homotopy commutative (and is hence bounded by weak equivalences in $(\Cati)_\Thomason$) by \cref{common sections and htpy equivces}\ref{sections are htpy equivces}, and
\item the square is by definition a pullback square in $\Cati$.
\end{itemize}
Our goal, then, reduces to showing that the commutative square in this diagram is also a homotopy pullback square in $(\Cati)_\Thomason$.  For this, it suffices to verify that in the induced commutative square
\[ \begin{tikzcd}
(F(\D) \da_n G(\E))^\gpd \arrow{r} \arrow{d}[swap]{t^\gpd} & (F(\D) \da_n \C)^\gpd \arrow{d}{t^\gpd} \\
\E^\gpd \arrow{r}[swap]{G^\gpd} & \C^\gpd
\end{tikzcd} \]
in $\S$, we obtain an equivalence on fibers over every point of $\E^\gpd$.

Now, observe first that any point $\pt_\S \ra \E^\gpd$ is represented by a map $\pt_\Cati \xra{e} \E$ in $(\Cati)_\Thomason$.  Moreover, we have just seen that in the resulting commutative diagram
\[ \begin{tikzcd}
(F(\D) \da_n G(e)) \arrow{r} \arrow{d} & (F(\D) \da_n G(\E)) \arrow{r} \arrow{d}[swap]{t} & (F(\D) \da_n \C) \arrow{d}{t} \\
\{ e \} \arrow[hook]{r} & \E \arrow{r}[swap]{G} & \C
\end{tikzcd} \]
in $\Cati$, both the left square and the outer rectangle are fiber inclusions which are moreover homotopy pullback squares in $(\Cati)_\Thomason$ (where we take $c = G(e) \in \C$).  So we do indeed obtain an equivalence on fibers over every point of $\E^\gpd$ in the above commutative square in $\S$, which completes the proof.
\end{proof}

\begin{rem}
As \cite[Theorem {\Bn} (6.2)]{DKS-hocommdiag} specializes to \cite[Theorem B]{QuillenAKTI} in the case that $n=1$, our Theorem {\Bn} (\ref{thm Bn}) also generalizes \cite[Theorem B for $\infty$-categories (5.3)]{Barwick-Q}.
\end{rem}

The following definition allows us to formulate a useful sufficient condition for a functor to have property {\Bn}.

\begin{defn}
We say that $\C \in \Cati$ has \bit{property {\Cnbit}} if every functor $\pt_\Cati \ra \C$ has property {\Bn}.
\end{defn}

The sufficient condition is then provided by the following main supporting result of this section, which we refer to as \bit{Theorem {\Cnbit}}.

\begin{thm}\label{thm Cn}
If $\C \in \Cati$ has property {\Cn}, then any functor $\D \xra{F} \C$ has property {\Bn}.
\end{thm}

\begin{proof}
We must show that for any map $c_1 \xra{\varphi} c_2$ in $\C$, the resulting functor
\[ (F(\D) \da_n c_1) \ra (F(\D) \da_n c_2) \]
is in $\bW_\Thomason \subset \Cati$.  Recall from \cref{lift to fctrs to fibns over other} that this functor can be considered as the image of the map $\varphi$ under a functor
\[ \C \xra{(F(\D) \da_n - )} \evencoCartFib(\D) . \]
Via the Grothendieck construction
\[ \Fun(\D^\oddop,\Cati) \xra[\sim]{\Gr} \evencoCartFib(\D) \]
this is classified by a natural transformation
\[ (F(-) \da_n c_1) \xra{\Gr^{-1}((F(\D) \da_n \varphi))} (F(-) \da_n c_2) \]
in $\Fun(\D^\oddop , \Cati)$, and by \cref{natural thomason weak equivalence induces a thomason weak equivalence on grothendieck constructions} it suffices to show that the components of this natural transformation lie in $\bW_\Thomason \subset \Cati$.  This is just the assertion that for any object $d \in \D$, the induced map
\[ (F(d) \da_n c_1) \ra (F(d) \da_n c_2) \]
is in $\bW_\Thomason \subset \Cati$, which follows by applying the definition of property {\Cn} to the functor $\pt_\Cati \xra{F(d)} \C$.
\end{proof}

\section{The Bousfield--Kan colimit formula}\label{section bousfield--kan}

In 1-category theory, there's an extremely useful formula which expresses an arbitrary colimit as a coequalizer of maps between coproducts (at least when the ambient category is cocomplete).  Namely, if we are given $\C , \D \in \strcat$, $\C$ admits coproducts and coequalizers, and $\D \xra{F} \C$ is any functor, then we have an isomorphism
\[ \colim_\D F \cong \coeq \left( \coprod_{(d_1 \ra d_2) \in \Nerve(\D)_1} F(d_1) \rra \coprod_{d \in \Nerve(\D)_0} F(d) \right) \]
in $\C$, where on the summand corresponding to $(d_1 \xra{\varphi} d_2) \in \Nerve(\D)_1$, one map is given by
\[ F(d_1) \xra{\id_{F(d_1)}} F(d_1) \]
(landing on the summand corresponding to $d_1 \in \Nerve(\D)_0$) and the other map is given by
\[ F(d_1) \xra{F(\varphi)} F(d_2) \]
(landing on the summand corresponding to $d_2 \in \Nerve(\D)_0$).  (In fact, it follows from this formula that $\C$ is cocomplete if (and only if) it admits coproducts and coequalizers.)

In this section, we generalize this colimit formula to $\infty$-categories.  Two things must be changed.  First of all, this coequalizer will be replaced by a geometric realization.\footnote{Recall that $\bD^{op}_{\leq 1}$ is the walking reflexive pair, so that colimits over it are precisely reflexive coequalizers.  In fact, the above pair of parallel arrows in $\C$ is indeed a reflexive pair: a common section is given by taking the summand $F(d)$ over the element $d \in \Nerve(\D)_0$ to summand $F(d)$ over the element $\id_d \in \Nerve(\D)_1$ via the identity map $\id_{F(d)}$ in $\C$.}  Moreover, the coproducts over the \textit{sets} of objects and morphisms of $\D$ will be replaced by the colimits over the \textit{spaces} of objects, morphisms, pairs of two composable morphisms, etc., of the diagram $\infty$-category -- in other words, over the constituents of its $\infty$-categorical nerve.  To be more precise, the \textit{simplicial replacement} of a diagram $\D \xra{F} \C$ in $\C$ will be a simplicial object $\srep(F)_\bullet \in s\C$ which in level $n$ is the colimit of the composite
\[ \Nervei(\D)_n \xra{\{0\}} \Nervei(\D)_0 \simeq \D^\simeq \hookra \D \xra{F} \C \]
(which of course only need exist when $\C$ has a sufficient supply of colimits).  Then, the geometric realization of the simplicial replacement will compute the colimit $\colim_\D F$ of the original diagram in $\C$.

This section is organized as follows.  In \cref{subsection BK formula}, we carefully construct the simplicial replacement and prove the \textit{Bousfield--Kan colimit formula} (\cref{bousfield--kan}).  In \cref{subsection examples of BK formula}, we provide some examples to illustrate the usage of this formula.  Finally, in \cref{subsection functoriality of BK formula}, we provide a functoriality result.

\begin{rem}
Of course, one can dualize this entire section to obtain analogous constructions and results concerning limits in a complete $\infty$-category.
\end{rem}

\subsection{The Bousfield--Kan colimit formula}\label{subsection BK formula}

In this subsection, we construct the simplicial replacement and prove the main result of this section.  We begin with the following observation.

\begin{rem}\label{rem lax nat trans for srep}
The structure maps of the diagrams in $\C$ assembling to the simplicial replacement are not strictly compatible: we will only have a lax natural transformation $\Nervei(\D)_\bullet \laxra \const(\D)$ in $\Fun(\bD^{op},\Cati)$, i.e.\! a map
\[ \begin{tikzcd}
\Gr(\Nervei(\D)_\bullet) \arrow[dashed]{rr} \arrow{rd} & & \D \times \bD^{op} \arrow{ld} \\
& \bD^{op}
\end{tikzcd} \]
making the triangle commute, which by precomposition will induce the passage to the simplicial replacement.  For instance, a point $\varphi \in \Nervei(\D)_1$ corresponds to a morphism $x \xra{\varphi} y$ in $\D$, and the above map should send this to the point $x \in \D \times \{ [1]^\opobj \}$.  On the other hand, the two simplicial structure maps take the point $\varphi \in \Nervei(\D)_1$ to the points $x,y \in \Nervei(\D)_0$, and the map $\Nervei(\D)_0 \ra \D \times \{ [0]^\opobj \}$ is simply the inclusion of the maximal subgroupoid.  So in the diagram
\[ \begin{tikzcd}
\Nervei(\D)_1 \arrow{r}{\{0\}} \arrow{d}[swap]{\delta_0} & \D \arrow{d}{\delta_0}[sloped, anchor=north]{\sim} \\
\Nervei(\D)_0 \arrow{r}[swap]{\{0\}} & \D ,
\end{tikzcd} \]
going across and then down gives $\varphi \mapsto x \mapsto x$, while going down and then across gives $\varphi \mapsto y \mapsto y$.  Hence, the diagram does not strictly commute.  However, it will commute \textit{up to a natural transformation} (running down and to the left), whose component at the point $\varphi \in \Nervei(\D)_1$ is simply the map $\varphi$ itself.
\end{rem}

\begin{rem}
In light of \cref{rem lax nat trans for srep}, one sees that the ``decomposition of colimits'' results of \sec T.4.2.3 do not suffice for our purposes here: they only apply to \textit{strict} diagrams of diagram $\infty$-categories lying over our diagram $\infty$-category $\D$.
\end{rem}

We now construct the lax natural transformation of \cref{rem lax nat trans for srep}.  In fact, we will construct it in a way which is functorial in $\D \in \Cati$.\footnote{This is essentially a homotopy-invariant analog of the ``first vertex projection'' from the opposite of the category of simplices of a quasicategory.  The main difficulty lies in keeping careful track of all coherence data.}

\begin{constr}\label{construct pre-srep map}
We construct the map $\Gr(\Nervei(\D)_\bullet) \ra \D \times \bD^{op}$ in $(\Cati)_{/\bD^{op}}$ as follows.  By the universal property of products, it suffices to construct only the map $\Gr(\Nervei(\D)_\bullet) \ra \D$ in $\Cati$.  This we construct on nerves, i.e.\! as a map $\Nervei(\Gr(\Nervei(\D)_\bullet))_\bullet \ra \Nervei(\D)_\bullet$ in $\CSS \subset s\S$.  For this, we will unwind the definitions of these two constructions as functors of $\D \in \Cati$.

Of course, the target is corepresented in level $n$ by $[n] \in \bD \subset \Cati$.

On the other hand, every string of composable morphisms in $\Gr(\Nervei(\D)_\bullet)$ is uniquely determined by its source and its image in $\bD^{op}$ (since the functor $\Gr(\Nervei(\D)_\bullet) \ra \bD^{op}$ is a left fibration and hence all maps are cocartesian), and so we obtain an equivalence
\[ \Nervei(\Gr(\Nervei(\D)_\bullet))_n \simeq \coprod_{\alpha \in \Nerve(\bD^{op})_n} \Nervei(\D)_{\alpha(0)} , \]
where we consider $\alpha \in \Nerve(\bD^{op})_n \cong \hom_\strcatsup([n],\bD^{op})$.  Hence, the composite functor
\[ \Cati \xra{\Nervei(-)_\bullet} s\S \simeq \Fun(\bD^{op},\S) \xra{\Gr} \LFib(\bD^{op}) \xra{\forget_{\L(\bD^{op})}} (\Cati)_{/\bD^{op}} \ra \Cati \xra{\Nervei(-)_\bullet} s\S , \]
considered by adjunction as a simplicial object in $\Fun(\Cati,\S)$, is given in each level by a coproduct of corepresentable objects.  In level $n$, it is given by
\[ \coprod_{\alpha \in \Nerve(\bD^{op})_n} \Yo(\alpha(0)^\opobj) , \]
where we write $\Yo = \Yo_{(\Cati)^{op}}$ for the contravariant Yoneda functor
\[ (\Cati)^{op} \xra{\hom_\Cati(-,=)} \Fun(\Cati,\S) \]
for brevity.
To describe its simplicial structure maps, given a map $[n] \xra{\alpha} \bD^{op}$, for any $0 \leq i \leq j \leq n$ let us denote the corresponding map in $\bD^{op}$ selected by $\alpha$ by
\[ \alpha(i) \xra{\alpha_{i,j}^\opobj} \alpha(j) \]
(i.e.\! the image under $\alpha$ of the unique element of $\hom_{[n]}(i,j)$).  Then, associated to a map $[n]^\opobj \xra{\varphi^\opobj} [m]^\opobj$ in $\bD^{op}$, the structure map
\[ \left( \coprod_{\alpha \in \Nerve(\bD^{op})_n} \Yo(\alpha(0)^\opobj) \right) \ra \left( \coprod_{\beta \in \Nerve(\bD^{op})_m} \Yo(\beta(0)^\opobj) \right) \]
of this simplicial object in $\Fun(\Cati,\S)$ is given by taking the summand $\Yo(\alpha(0)^\opobj)$ indexed by $\alpha \in \hom_\strcatsup([n],\bD^{op})$ to the summand $\Yo(\beta(0)^\opobj)$ indexed by $\beta = \alpha \circ \varphi \in \hom_\strcatsup([m],\bD^{op})$ via the map
\[ \Yo(\alpha(0)^\opobj) \xra{\Yo(\alpha_{0,\varphi(0)}^\opobj)} \Yo(\alpha(\varphi(0))^\opobj) \]
in $\Fun(\Cati,\S)$.  Since the objects of $\Fun(\Cati,\S)$ corepresented by gaunt categories and their finite coproducts generate a full subcategory which is just a 1-category, no higher coherence issues arise.  (We will implicitly appeal to this fact for our further manipulations as well.)

We can now describe our desired map $\Nervei(\Gr(\Nervei(\D)_\bullet))_\bullet \ra \Nervei(\D)_\bullet$ in $\CSS \subset s\S$.  In level $n$, it is obtained from the map
\[ \left( \coprod_{\alpha \in \Nerve(\bD^{op})_n} \Yo(\alpha(0)^\opobj) \right) \ra \Yo([n]^\opobj) \]
which on the summand $\alpha \in \Nerve(\bD^{op})_n$ is the map
\[ \Yo \left( \alpha(0)^\opobj \xra{ ( ( i \mapsto \alpha_{0,i}(0))_{0 \leq i \leq n})^\opobj} [n]^\opobj \right) \]
in $\Fun(\Cati,\S)$.  That is, in level $n$, on the summand $\alpha$ it is corepresented by the map $[n] \ra \alpha(0)^\opobj$ in $\bD \subset \Cati$ given by $i \mapsto \alpha_{0,i}(0)$, i.e.\! the map taking $i \in [n]$ to the image in $\alpha(0)^\opobj$ of the object $0 \in \alpha(i)^\opobj$ under the composite
\[ \alpha(0)^\opobj \xla{\alpha_{0,1}} \cdots \xla{\alpha_{i-1,i}} \alpha(i)^\opobj \]
in $\bD$.

We have now associated to the object $\D \in \Cati$ a map $\Nervei(\Gr(\Nervei(\D)_\bullet)) \ra \Nervei(\D)$ in $\CSS$.  In fact, this map clearly commutes with the induced projections to $\Nervei(\bD^{op})$, and moreover, this association is clearly functorial in $\D \in \Cati$ since it is entirely corepresented.  Hence, we obtain a functor
\[ \Cati \ra \Fun([1],\CSS_{/\Nervei(\bD^{op})}) , \]
which immediately (and equivalently) produces our desired functor
\[ \Cati \ra \Fun([1],(\Cati)_{/\bD^{op}}) \]
which takes the object $ \D \in \Cati$ to a commutative triangle as depicted in \cref{rem lax nat trans for srep} (considered as an object of $\Fun([1],(\Cati)_{/\bD^{op}})$).
\end{constr}

\begin{ex}
To illustrate the combinatorics of the corepresenting map in \cref{construct pre-srep map}, we return to the situation described in \cref{rem lax nat trans for srep}.  Namely, let us restrict our attention to the case that $n=1$, $\alpha(0) = [1]^\opobj$, and $\alpha(1) = [0]^\opobj$.  Then, there are two possibilities for the map $\alpha \in \Nerve(\bD^{op})_1 = \hom_\strcatsup([1],\bD^{op})$: it selects either $[1]^\opobj \xra{\delta_0} [0]^\opobj$ or $[1]^\opobj \xra{\delta_1} [0]^\opobj$.  These are respectively opposite to the map $[0] \ra [1]$ in $\bD$ given by $0 \mapsto 1$ or $0 \mapsto 0$.  Hence, the corresponding corepresenting map $[1] \ra \alpha(0)^\opobj = [1]$ in $\bD$ is respectively either $\id_{[1]}$ or $\const(0)$.
\end{ex}

We now define our main object of interest, the simplicial replacement.

\begin{defn}\label{defn simp repl}
Suppose that $\C$ is cocomplete, and let $\D \xra{F} \C$ be a functor.  We construct a composite
\[ \begin{tikzcd}
\Gr(\Nervei(\D)_\bullet) \arrow{rr} \arrow{rrd} & & \D \times \bD^{op} \arrow{rr}{F \times \id_{\bD^{op}}} \arrow{d} & & \C \times \bD^{op} \arrow{lld} \\
& & \bD^{op}
\end{tikzcd} \]
in which the first horizontal map is given by \cref{construct pre-srep map}.  By Proposition T.4.2.2.7, there is a unique lift in the diagram
\[ \begin{tikzcd}
\Gr(\Nervei(\D)_\bullet) \arrow{r} \arrow{d} & \C \times \bD^{op} \arrow{d} \\
\Gr(\Nervei(\D)_\bullet) \underset{\bD^{op}}{\diamond} \bD^{op} \arrow{r} \arrow[dashed]{ru} & \bD^{op}
\end{tikzcd} \]
such that the composite
\[ \bD^{op} \ra \Gr(\Nervei(\D)_\bullet) \underset{\bD^{op}}{\diamond} \bD^{op} \ra \C \times \bD^{op} \]
takes each $[n]^{op} \in \bD^{op}$ to a colimit of the composite
\[ \Nervei(\D)_n \xra{\{0\}} \Nervei(\D)_0 \simeq \D^\simeq \hookra \D \xra{F} \C . \]
We define the object $\srep(F)_\bullet \in s\C$ to be the composite
\[ \bD^{op} \ra \Gr(\Nervei(\D)_\bullet) \underset{\bD^{op}}{\diamond} \bD^{op} \ra \C \times \bD^{op} \ra \C , \]
and refer to it as the \bit{simplicial replacement} of the functor $F$.
\end{defn}

\begin{lem}\label{srep functorial in target for cocts fctrs}
Suppose that $\C_1 \xra{\chi} \C_2$ is a cocontinuous functor between cocomplete $\infty$-categories, and suppose that $\D \xra{F} \C_1$ is any functor.  Then the composite
\[ \bD^{op} \xra{\srep(F)_\bullet} \C_1 \xra{\chi} \C_2 \]
is canonically equivalent to the object $\srep(\chi \circ F)_\bullet \in s(\C_2)$.
\end{lem}

\begin{proof}
This follows directly from \cref{defn simp repl}.
\end{proof}

We can now give the main result of this section, the \bit{Bousfield--Kan colimit formula}.

\begin{thm}\label{bousfield--kan}
Let $\C$ be cocomplete, and let $\D \xra{F} \C$ be a functor.  Then there is a canonical equivalence
\[ \colim_\D F \simeq | \srep(F)_\bullet | \]
in $\C$.
\end{thm}

\begin{proof}
Observe that the colimit of a $\D$-shaped diagram is functorial in cocomplete $\infty$-categories under $\D$ and cocontinuous functors between them.  Hence, in light of \cref{srep functorial in target for cocts fctrs}, it suffices to prove the statement in the universal (i.e.\! initial) case, namely taking the functor $\D \xra{F} \C$ to be the canonical map from $\D$ into its free cocompletion, i.e.\! the Yoneda embedding $\D \xra{\Yo} \P(\D) = \Fun(\D^{op},\S)$.

First of all, we claim that there is a canonical equivalence $\colim_\D(\Yo) \simeq \const(\pt_\S)$ in $\P(\D)$.  To see this, observe that for any $d \in \D$, we have a string of equivalences
\[ (\colim_\D \Yo)(d) \simeq \colim_\D(\ev_d \circ \Yo) \simeq \colim_\D(\hom_\D(d,-)) \]
in $\S$, where the first equivalence is because colimits in $\P(\D)$ are computed pointwise and the second is simply because $\ev_d \circ \Yo \simeq \hom_\D(d,-)$ in $\Fun(\D,\S)$.  Appealing to \cref{groupoid-completion of the grothendieck construction} and the canonical equivalence
\[ \Gr ( \hom_\D(d,-) ) \simeq \D_{d/} \]
in $\LFib(\D)$ of \cref{slice cats are l/r fibns}, we obtain a string of equivalences
\[ \colim_\D(\hom_\D(d,-)) \simeq \Gr( \hom_\D(d,-) )^\gpd \simeq (\D_{d/})^\gpd \simeq \pt_\S , \]
where the last equivalence follows from the dual of \cref{groupoid-completion of infty-cat with terminal object is contractible}.  Hence, $(\colim_\D \Yo)(d) \simeq \pt_\S$ for every $d \in \D$, and so it follows that the terminal map $\colim_\D (\Yo) \ra \const(\pt_\S)$ in $\P(\D)$ is indeed an equivalence.\footnote{This is also proved as Lemma T.5.3.3.2.}

On the other hand, recall that $\srep(\Yo)_n \in \P(\D)$ is the colimit of the composite
\[ \Nervei(\D)_n \xra{\{0\}} \Nervei(\D)_0 \simeq \D^\simeq \hookra \D \xra{\Yo} \P(\D) . \]
For any $d \in \D$, the composite functor
\[ \Nervei(\D)_n \xra{\{0\}} \Nervei(\D)_0 \simeq \D^\simeq \hookra \D \xra{\Yo} \P(\D) \xra{\ev_d} \S \]
classifies the left fibration which is the upper composite map
\[ \begin{tikzcd}
\{ d \} \underset{\Nervei(\D)_0 , \{0\}}{\times} \Nervei(\D)_{n+1} \arrow{r} \arrow{d} & \Nervei(\D)_{n+1} \arrow{r}{\delta_0} \arrow{d}{\{0\}} & \Nervei(\D)_n \\
\{ d \} \arrow[hook]{r} & \Nervei(\D)_0 .
\end{tikzcd} \]
Again appealing to \cref{groupoid-completion of the grothendieck construction} (and the fact that colimits in $\P(\D)$ are computed pointwise), it follows that
\begin{align*}
\srep(\Yo)_n(d)
& \simeq \left( \colim \left( \Nervei(\D)_n \xra{\{0\}} \Nervei(\D)_0 \simeq \D^\simeq \hookra \D \xra{\Yo} \P(\D) \right) \right) (d) \\
& \simeq \colim \left( \Nervei(\D)_n \xra{\{0\}} \Nervei(\D)_0 \simeq \D^\simeq \hookra \D \xra{\Yo} \P(\D) \xra{\ev_d} \S \right) \\
& \simeq \left( \{ d \} \underset{ \Nervei(\D)_0 , \{0\}}{\times} \Nervei(\D)_{n+1} \right)^\gpd \\
& \simeq \{ d \} \underset{\Nervei(\D)_0 , \{0\}}{\times} \Nervei(\D)_{n+1}
\end{align*}
(where the last equivalence follows from the fact that the inclusion $\S \subset \Cati$ is a right adjoint and hence commutes with pullbacks (and the fact that $(-)^\gpd : \Cati \ra \S$ is idempotent)).  Unwinding the definitions of the simplicial structure maps of $\srep(\Yo)_\bullet$, we obtain that in fact
\[ \srep(\Yo)_\bullet(d) \simeq \left( \{ d \} \underset{\Nervei(\D)_0 , \{0\}}{\times} \Nervei(\D)_{\bullet + 1} \right) \simeq \Nervei(\D_{d/})_\bullet . \]
Hence, it follows that
\[ |\srep(\Yo)_\bullet|(d) \simeq | \srep(\Yo)_\bullet(d)| \simeq |\Nervei(\D_{d/})_\bullet| \simeq (\D_{d/})^\gpd \simeq \pt_\S , \]
where again the last equivalence follows from the dual of \cref{groupoid-completion of infty-cat with terminal object is contractible}, and the second-to-last equivalence follows from \cref{rnerves:groupoid-completion of CSSs}.  Therefore, the terminal map $|\srep(\Yo)_\bullet| \ra \const(\pt_\S)$ in $\P(\D)$ is also an equivalence.  It follows that we have a canonical equivalence
\[ \colim_\D \Yo \simeq \const(\pt_\S) \simeq | \srep(\Yo)_\bullet | , \]
which proves the claim.
\end{proof}

\begin{rem}
Our Bousfield--Kan colimit formula (\cref{bousfield--kan}) is certainly inspired by the classical Bousfield--Kan formula (see the original source \cite[Chapter XII, \sec 5]{BousKan-yellow}, or e.g.\! \cite[\sec 18.1]{Hirsch} for a more modern treatment), but it is actually rather different: the latter computes a \textit{homotopy} colimit in a simplicial model category.  Moreover, even in the case that both $\C$ and $\D$ are only 1-categories, it does not generally agree with the formula for a colimit as a coequalizer of coproducts: this additionally requires that $\D$ be \textit{gaunt}.  (Note that that formula is actually evil: it is not invariant under replacing $\D$ by an equivalent category, referring as it does to its actual sets $\Nerve(\D)_0$ and $\Nerve(\D)_1$ of objects and of morphisms.)  On the other hand, when $\D$ is a (1- or $\infty$-)groupoid, then our simplicial replacement (\cref{defn simp repl}) is trivial: in that case we have a canonical equivalence $\Nervei(\D)_\bullet \simeq \const(\D)$ in $s\C$, whence it follows that the simplicial replacement $\srep(F)_\bullet \in s\C$ of a diagram $\D \xra{F} \C$ is already constant at the object $\colim_\D F \in \C$.
\end{rem}

\subsection{Examples of the Bousfield--Kan colimit formula}\label{subsection examples of BK formula}

The Bousfield--Kan colimit formula (\cref{bousfield--kan}) is most interesting (and novel) when the diagram $\infty$-category $\D$ is not merely a 1-category.  For instance, applying it to a pushout diagram and canceling out the redundancies in the resulting geometric realization yields nothing but the original pushout diagram.  We therefore give two inherently $\infty$-categorical examples to illustrate its application.

\begin{ex}
Choose any space $Y \in \S$, and let us take $\D$ to be its ``categorical suspension'': this is an $\infty$-category with two objects $d_1$ and $d_2$, which is determined by the prescriptions that
\begin{itemizesmall}
\item $\hom_\D(d_1,d_1) \simeq \hom_\D(d_2,d_2) \simeq \pt_\S$,
\item $\hom_\D(d_2,d_1) \simeq \es_\S$, and
\item $\hom_\D(d_1,d_2) \simeq Y$.
\end{itemizesmall}
Then, a functor $\D \xra{F} \C$ selects the data of
\begin{itemizesmall}
\item a pair of objects $c_1,c_2 \in \C$, and
\item a map $Y \ra \hom_\S(c_1,c_2)$ in $\S$.
\end{itemizesmall}
Canceling out redundancies (and assuming $\C$ is cocomplete), the Bousfield--Kan colimit formula (\cref{bousfield--kan}) then gives equivalences
\[ \colim_\D(F) \simeq | \srep(F)_\bullet| \simeq \colim \left( \begin{tikzcd}
c_1 \tensoring Y \arrow{r} \arrow{d} & c_1 \\
c_2
\end{tikzcd} \right) \]
in $\C$, where in the pushout
\begin{itemizesmall}
\item the horizontal map is given by $c_1 \tensoring (Y \ra \pt_\S)$, and
\item the vertical map is the adjunct of the chosen map $Y \ra \hom_\C(c_1,c_2)$.
\end{itemizesmall}
\end{ex}

\begin{ex}
Choose any space $Y \in \S$, and let us take $\D = Y^\leftcone \in \Cati$ to be the left cone on it (considered as an $\infty$-groupoid).  Assume for simplicity that $\C$ is bicomplete.  Then, a functor $\D \xra{F} \C$ is determined by
\begin{itemizesmall}
\item its restriction $Y \xra{F'} \C$, and
\item a map $c \ra \lim_Y(F')$ from some object $c \in \C$ (selected by the cone point) to the limit of this restriction.\footnote{If $Y \xra{F'} \C$ is constant, then $\lim_Y(F')$ reduces to a cotensor.  However, there are interesting examples where this is not the case: for instance, if $Y$ is a 1-type, this limit computes the \textit{fixed points} of the corresponding group action.  Of course, a dual observation applies to $\colim_Y(F')$, which appears here as well.  (Another interesting example of a colimit over an $\infty$-groupoid is the \textit{Thom spectrum} construction.)}
\end{itemizesmall}
Canceling out redundancies, the Bousfield--Kan colimit formula (\cref{bousfield--kan}) then gives equivalences
\[ \colim_\D(F) \simeq | \srep(F)_\bullet | \simeq \colim \left( \begin{tikzcd}
c \tensoring Y \arrow{r} \arrow{d} & c \\
\colim_Y(F')
\end{tikzcd} \right) \]
in $\C$, where in the pushout
\begin{itemizesmall}
\item the horizontal map is given by $c \tensoring (Y \ra \pt_\S)$, and
\item the vertical map is induced by the natural transformation $\const_Y(c) \ra F'$ in $\Fun(Y,\C)$ which is adjunct to the chosen map $c \ra \lim_Y(F')$.
\end{itemizesmall}
On an object $c' \in \C$, this pushout corepresents the data of
\begin{itemizesmall}
\item a morphism $\colim_Y(F') \ra c'$ in $\C$ (or equivalently, a morphism $F' \ra \const_Y(c')$ in $\Fun(Y,\C)$, along with
\item a trivialization
\[ \begin{tikzcd}
Y \arrow{r} \arrow{d} & \hom_\C(c,c') \\
\pt_\S \arrow[dashed]{ru}
\end{tikzcd} \]
of the adjunct to the induced composite
\[ c \tensoring Y \ra \colim_Y(F') \ra c' . \]
\end{itemizesmall}
\end{ex}

\subsection{Functoriality of the Bousfield--Kan colimit formula}\label{subsection functoriality of BK formula}

Of course, it is perfectly reasonable to expect that the Bousfield--Kan colimit formula enjoys good functoriality properties, along the lines of those explored in \cref{subsection global co/limit}.  However, rather than pursue a full treatment, in this subsection we exhibit only the mere shadow of such functoriality that we will actually need.

We begin by identifying the simplicial replacement as a left Kan extension.

\begin{lem}\label{srep is a left kan extn}
Suppose that $\C$ is cocomplete, let $\D \xra{F} \C$ be any diagram.  Then, in the commutative diagram
\[ \begin{tikzcd}
& \Gr(\Nervei(\D)_\bullet) \arrow{r} \arrow{d} & \C \times \bD^{op} \arrow{r} & \C \\
\bD^{op} \arrow{r} & \Gr(\Nervei(\D)_\bullet) \underset{\bD^{op}}{\diamond} \bD^{op} \arrow{ru} \arrow[dashed, bend right=15]{rru}
\end{tikzcd} \]
containing the simplicial replacement $\bD^{op} \xra{\srep(F)_\bullet} \C$ as a composite, the vertical functor is a full inclusion and the dotted arrow is a left Kan extension along it.
\end{lem}

\begin{proof}
To see that the vertical functor is a full inclusion, we simply unwind the definition of its target to obtain
\begin{align*}
& \Gr(\Nervei(\D)_\bullet) \underset{\bD^{op}}{\diamond} \bD^{op} \\
& = \Gr(\Nervei(\D)_\bullet)
\coprod_{\left( \Gr(\Nervei(\D)_\bullet) \underset{\bD^{op}}{\times} \bD^{op} \times \{0\} \right)}
\left( \Gr(\Nervei(\D)_\bullet) \underset{\bD^{op}}{\times} \bD^{op} \times [1] \right)
\coprod_{\left( \Gr(\Nervei(\D)_\bullet) \underset{\bD^{op}}{\times} \bD^{op} \times \{1\} \right)}
\bD^{op} \\
& \simeq \left( \Gr(\Nervei(\D)_\bullet) \times [1] \right) \coprod_{ \left( \Gr(\Nervei(\D)_\bullet) \times \{1\} \right) } \bD^{op} .
\end{align*}
That is, this target is precisely the cocartesian fibration over $[1]$ classified by the functor $[1] \ra \Cati$ which selects the projection $\Gr(\Nervei(\D)_\bullet) \ra \bD^{op}$, and our vertical functor is the fiber inclusion over the object $0 \in [1]$; in particular, it is full (as the object $0 \in [1]$ admits no nontrivial automorphisms).

Now, to check that the dotted arrow is a left Kan extension along this full inclusion, by Remark T.4.3.2.3 it suffices to show that for any object
\[ [n]^\opobj \in \bD^{op} \subset \left( \Gr(\Nervei(\D)_\bullet) \underset{\bD^{op}}{\diamond} \bD^{op} \right) , \]
the corresponding fiber inclusion $\Nervei(\D)_n \hookra \Gr(\Nervei(\D)_\bullet)$ induces a functor
\[ \Nervei(\D)_n \ra \left( \Gr(\Nervei(\D)_\bullet) \underset{\left( \Gr(\Nervei(\D)_\bullet) \underset{\bD^{op}}{\diamond} \bD^{op} \right)}{\times} \left( \Gr(\Nervei(\D)_\bullet) \underset{\bD^{op}}{\diamond}\bD^{op} \right)_{/[n]^\opobj} \right) \]
which is final.  This is straightforwardly verified using Theorem {\thmA} (\ref{theorem A}): all of the comma $\infty$-categories whose groupoid completions must be shown to be contractible are easily seen to possess initial objects, and hence the equivalent condition follows from the opposite of \cref{groupoid-completion of infty-cat with terminal object is contractible}.
\end{proof}

We can now describe our desired shadow of functoriality.

\begin{prop}\label{bousfield--kan functorial in source for Cati/C}
Let $\C$ be cocomplete, and suppose that
\[ \begin{tikzcd}
\D \arrow{rr}{H} \arrow{rd}[swap]{F} & & \E \arrow{ld}{G} \\
& \C
\end{tikzcd} \]
is a commutative diagram in $\Cati$.
\begin{enumerate}
\item\label{map on sreps}
There is a canonical induced map
\[ \srep(F)_\bullet \ra \srep(G)_\bullet \]
in $s\C$, which is functorial in the variable $\C$ for cocontinuous functors between cocomplete $\infty$-categories.
\item\label{compatibility of map on sreps}
We have a commutative square
\[ \begin{tikzcd}
\colim_\D F \arrow{r} \arrow{d}[sloped, anchor=north]{\sim} & \colim_\E G \arrow{d}[sloped, anchor=south]{\sim} \\
| \srep(F)_\bullet | \arrow{r} & | \srep(G)_\bullet |
\end{tikzcd} \]
in $\C$, in which
\begin{itemize}
\item the upper map is the induced map on colimits of the global colimit functor (\cref{global colimit}),
\item the lower map is the geometric realization of the canonical map of part \ref{map on sreps}, and
\item the vertical equivalences are those of \cref{bousfield--kan}.
\end{itemize}
\end{enumerate}
\end{prop}

\begin{proof}
For part \ref{map on sreps}, by \cref{srep is a left kan extn} and the functoriality provided by \cref{construct pre-srep map}, we have a commutative diagram
\[ \begin{tikzcd}
\Gr(\Nervei(\D)_\bullet) \arrow{rd} \arrow[hook]{dd} \arrow{rrr} & & & \C \arrow[leftarrow, dashed, bend left=20]{llldd} \\
& \Gr(\Nervei(\E)_\bullet) \arrow{rru} \arrow[hook, crossing over]{dd} \\
\Gr(\Nervei(\D)_\bullet) \underset{\bD^{op}}{\diamond} \bD^{op} \arrow{rd} \\ 
& \Gr(\Nervei(\E)_\bullet) \underset{\bD^{op}}{\diamond} \bD^{op} \arrow[dashed, bend right=20]{rruuu}
\end{tikzcd} \]
in which the vertical arrows are full inclusions and the dotted arrows are left Kan extensions therealong.  As the composite
\[ \left( \Gr(\Nervei(\D)_\bullet) \underset{\bD^{op}}{\diamond} \bD^{op} \right) \ra \left( \Gr(\Nervei(\E)_\bullet) \underset{\bD^{op}}{\diamond} \bD^{op} \right) \ra \C \]
also extends the map $\Gr(\Nervei(\D)_\bullet) \ra \C$ along the given full inclusion, it therefore admits a canonical natural transformation from the dotted map
\[ \Gr(\Nervei(\D)_\bullet) \underset{\bD^{op}}{\diamond} \bD^{op} \ra \C . \]
Restricting to the full subcategory
\[ \bD^{op} \subset \left( \Gr(\Nervei(\D)_\bullet) \underset{\bD^{op}}{\diamond} \bD^{op} \right), \]
we obtain a natural transformation
\[ \begin{tikzcd}
\bD^{op} \arrow[bend left=50]{r}{\srep(F)_\bullet}[swap, transform canvas={yshift=-0.8em}]{\Downarrow} \arrow[bend right=50]{r}[swap]{\srep(G)_\bullet} & \C
\end{tikzcd} \]
which is precisely our desired map in $s\C$.  Moreover, the asserted functoriality follows easily from this argument (recall \cref{functoriality of global colim}).

For part \ref{compatibility of map on sreps}, note that there is a unique cocontinuous functor $\P(\D) \ra \P(\E)$ making the diagram
\[ \begin{tikzcd}
\D \arrow{r}{H} \arrow{d}[swap]{\Yo_\D} & \E \arrow{d}{\Yo_\E} \\
\P(\D) \arrow[dashed]{r} & \P(\E)
\end{tikzcd} \]
commute.\footnote{In fact, by Lemma T.5.1.5.5(1), this functor must be precisely the left Kan extension $(\Yo_\D)_! ( \Yo_\E \circ H)$.}  From here, by \cref{srep functorial in target for cocts fctrs} and the functoriality asserted in part \ref{map on sreps}, it suffices to verify the claim in the case that the functor $\E \xra{G} \C$ is the Yoneda embedding $\E \xra{\Yo_\E} \P(\E)$, so that the functor $\D \xra{F} \C$ is the composite $\D \xra{\Yo_\E \circ H} \P(\E)$.  Hence, it remains to show that we have a commutative square
\[ \begin{tikzcd}
\colim_\D (\Yo_\E \circ H) \arrow{r} \arrow{d}[sloped, anchor=north]{\sim} & \colim_\E \Yo_\E \arrow{d}[sloped, anchor=south]{\sim} \\
| \srep(\Yo_\E \circ H)_\bullet | \arrow{r} & | \srep(\Yo_\E)_\bullet |
\end{tikzcd} \]
in $\P(\E)$ satisfying the described criteria.  But as we have seen in the proof of \cref{bousfield--kan}, if we consider these two vertical maps as objects of $\Fun([1],\P(\E))$, then the one on the right determines a terminal object.  As the datum of this commutative square is equivalent to that of the corresponding morphism in $\Fun([1],\P(\E))$ (reading the square from left to right), it follows that in fact the square must commute in a canonical and unique way.
\end{proof}

\appendix

\section{The Thomason model structure on the $\infty$-category of $\infty$-categories}\label{section thomason}

In this appendix, we equip the $\infty$-category $\Cati$ of $\infty$-categories with a \textit{Thomason model structure} analogous to the classical Thomason model structure on $\strcat$ (though see \cref{different Thomasons}) and observe some of its basic features.  (We refer the reader to \cref{sspaces:section model infty-cats defns} for the definition of a model structure on an $\infty$-category, and to \cref{sspaces:section define kan--quillen model structure} for the definition of the Kan--Quillen model structure on the $\infty$-category $s\S$ of simplicial spaces.)  This model structure provides a convenient language for a number of the results in the main body of the paper.

We begin by constructing it.

\begin{thm}\label{Thomason model str on CSS}
The Kan--Quillen model structure on $s\S$ lifts along the composite adjunction
\[ \begin{tikzcd}[column sep=1.5cm]
s\S 
{\arrow[transform canvas={yshift=0.7ex}]{r}{\leftloc_\CSS}[swap, transform canvas={yshift=0.25ex}]{\scriptstyle \bot} \arrow[transform canvas={yshift=-0.7ex}, hookleftarrow]{r}[swap]{\forget_\CSS}}
& \CSS
{\arrow[transform canvas={yshift=0.7ex}]{r}{\Nervei^{-1}}[swap, transform canvas={yshift=0.05ex}]{\sim} \arrow[transform canvas={yshift=-0.7ex}, leftarrow]{r}[swap]{\Nervei}}
& \Cati .
\end{tikzcd} \]
Moreover, this Quillen adjunction is a Quillen equivalence.
\end{thm}

\begin{proof}
For notational convenience, we prove the statement for the adjunction $\leftloc_\CSS \adj \forget_\CSS$ (which is of course equivalent).

For the first claim, we verify the hypotheses of the lifting theorem for cofibrantly generated model $\infty$-categories (\Cref{sspaces:lift cofgen}) in turn.

First of all, for any $Y \in s\S$ and $Z \in \CSS$ we have that
\[ \hom_\CSS(\leftloc_\CSS(Y),Z) \simeq \hom_{s\S}(Y,\forget_\CSS(Z)) , \]
and so the fact that the sets $\leftloc_\CSS(I_\KQ)$ and $\leftloc_\CSS(J_\KQ)$ of homotopy classes of maps in $\CSS$ permit the small object argument follows from \cref{sspaces:finite ssets are small sspaces}.

Next, we show that the right adjoint $s\S \xhookla{\forget_\CSS} \CSS$ takes relative $\leftloc_\CSS(J_\KQ)$-cell complexes into $\bW_\KQ$.  For this, let us begin by supposing that
\[ \begin{tikzcd}
\leftloc_\CSS(\Lambda^n_i) \arrow{r} \arrow{d} & Y \arrow{d} \\
\leftloc_\CSS(\Delta^n) \arrow{r} & Z
\end{tikzcd} \]
is a pushout square in $\CSS$.  Since $s\S \xra{\leftloc_\CSS} \CSS$ commutes with pushouts (being a left adjoint), we have that
\[ Z \simeq \colim^\CSS \left(
\begin{tikzcd}
\leftloc_\CSS(\Lambda^n_i) \arrow{r} \arrow{d} & \leftloc_\CSS(\forget_\CSS(Y)) \\
\leftloc_\CSS(\Delta^n)
\end{tikzcd} \right) \simeq
\leftloc_\CSS \left( \colim^{s\S} \left(
\begin{tikzcd}
\Lambda^n_i \arrow{r} \arrow{d} & \forget_\CSS(Y) \\
\Delta^n
\end{tikzcd} \right) \right) .
\]
Using \cref{rnerves:groupoid-completion of CSSs} and the fact that colimits commute with colimits, we then obtain the string of equivalences
\begin{align*}
|\forget_\CSS(Z) | & \simeq \left| \forget_\CSS \left( \leftloc_\CSS \left( \colim^{s\S} \left(
\begin{tikzcd}[ampersand replacement=\&]
\Lambda^n_i \arrow{r} \arrow{d} \& \forget_\CSS(Y) \\
\Delta^n
\end{tikzcd} \right) \right) \right) \right| \\
& \simeq
\left| \colim^{s\S} \left(
\begin{tikzcd}[ampersand replacement=\&]
\Lambda^n_i \arrow{r} \arrow{d} \& \forget_\CSS(Y) \\
\Delta^n
\end{tikzcd} \right) \right| \\
& \simeq
\colim^\S \left(
\begin{tikzcd}[ampersand replacement=\&]
|\Lambda^n_i| \arrow{r} \arrow{d}[sloped, anchor=north]{\sim} \& |\forget_\CSS(Y)| \\
|\Delta^n|
\end{tikzcd} \right) \\
& \simeq |\forget_\CSS(Y)| .
\end{align*}
Hence, the map
\[ \forget_\CSS(Y) \ra \forget_\CSS(Z) \]
lies in $\bW_\KQ \subset s\S$.  Now, to prove the claim for more general transfinite compositions, it then suffices to observe that the composite
\[ \CSS \xhookra{\forget_\CSS} s\S \xra{|{-}|} \S \]
is a left adjoint and hence commutes with colimits.  Thus, $\forget_\CSS$ does indeed take relative $\leftloc_\CSS(J_\KQ)$-cell complexes into $\bW_\KQ$.

So, the Kan--Quillen model structure $s\S_\KQ$ does indeed lift along the adjunction as claimed, and it remains to check that the resulting Quillen adjunction is a Quillen equivalence.  For this, suppose we are given \textit{any} objects $Y \in s\S$ and $Z \in \CSS$.  Then, a map
\[ Y \ra \forget_\CSS(Z) \]
in $s\S$ corresponds via the adjunction to the map
\[ \leftloc_\CSS(Y) \ra \leftloc_\CSS(\forget_\CSS(Z)) \simeq Z \]
in $\CSS$.  Since the lifting theorem (\Cref{sspaces:lift cofgen}) produces a model structure for which the right adjoint creates the weak equivalences, the claim follows from \cref{rnerves:groupoid-completion of CSSs}.
\end{proof}

\begin{defn}\label{define Thomason model str}
We refer to the model structure on $\Cati$ defined by \cref{Thomason model str on CSS} as the \bit{Thomason model structure}, denoted $(\Cati)_\Thomason$.
\end{defn}

\begin{rem}\label{Th w.e.'s created by gpd compln}
\cref{rnerves:groupoid-completion of CSSs} implies that the subcategory $\bW_\Thomason \subset \Cati$ is created by the groupoid completion functor $(-)^\gpd : \Cati \ra \S$.  As this functor is a left localization, it therefore induces an equivalence $\loc{\Cati}{\bW_\Thomason} \xra{\sim} \S$.  Moreover, it is not hard to see directly that the adjoint functors $\Nervei^{-1} \circ \leftloc_\CSS \adj \forget_\CSS \circ \Nervei$ induce inverse equivalences $\loc{s\S}{\bW_\KQ} \simeq \loc{\Cati}{\bW_\Thomason}$.
\end{rem}

\begin{rem}\label{Thomason fibcy}
Let us explore what it means for an object $\C \in \Cati$ to be fibrant in the Thomason model structure.  By definition, this means that $\Nervei(\C) \in \CSS \subset s\S$ has the extension property for the set $J_\KQ = \{ \Lambda^n_i \ra \Delta^n \}_{0 \leq i \leq n \geq 1}$ of horn inclusions in $s\Set \subset s\S$.

In fact, the Segal condition on a simplicial space implies that it admits \textit{unique} fillers for the set $\{ \Lambda^n_i \ra \Delta^n \}_{0 < i < n \geq 2}$ of \textit{inner} horn inclusions.  To see this, observe that the inclusion
\[ \left( \Delta^{\{0,1\}} \coprod_{\Delta^{\{1\}}} \cdots \coprod_{\Delta^{\{n-1\}}} \Delta^{\{n-1,n\}} \right) \ra \Lambda^n_i \]
of subobjects of $\Delta^n \in s\Set$ can be constructed as a (finite) composition of pushouts of inner horn inclusions $\Lambda^k_j \ra \Delta^k$ for $k < n$, and the Segal condition is precisely the assertion that a given simplicial space admits unique extensions for the inclusion of the above source into $\Delta^n$ (the ``$n\th$ spine inclusion'').  From here, the assertion follows by induction.

On the other hand, it is already clear that if a (complete) Segal space has the extension property for the outer horn inclusion $\Lambda^2_0 \ra \Delta^2$, then it must be the nerve of an $\infty$-category whose morphisms are all invertible.  Hence, the fibrant objects in the model structure on $\CSS$ of \cref{Thomason model str on CSS} are precisely the \textit{constant} complete Segal spaces (i.e.\! those that are constant as simplicial spaces).  It follows that we can identify the subcategory of fibrant objects in the Thomason model structure as $(\Cati)^f_\Thomason = \S \subset \Cati$, the subcategory of $\infty$-groupoids (i.e.\! spaces).
\end{rem}

In fact, we have the following strengthening of \cref{Thomason fibcy}.\footnote{\cref{Thomason is left localizn} immediately implies the conclusion of \cref{Thomason fibcy} (that $(\Cati)^f_\Thomason = \S \subset \Cati$), but we will want to build on the explicit geometric arguments given there in \cref{Thomason fibns}.}\footnote{\cref{Thomason is left localizn} does \textit{not} follow from the discussion of \cref{sspaces:left localizations give model structures}.  For instance, this left localization is certainly not left exact; the entire point of \cref{section ho-p.b.'s in Thomason} is to obtain sufficient conditions under which it commutes with pullbacks.  See also \cref{Thomason acyclic fibns are all equivces}.}

\begin{prop}\label{Thomason is left localizn}
The Thomason model structure on $\Cati$ is a left localization model structure (in the sense of \cref{sspaces:left localizations give model structures}) with respect to the left localization adjunction
\[ \begin{tikzcd}[column sep=2cm]
\Cati \arrow[transform canvas={yshift=0.7ex}]{r}{(-)^\gpd}[swap, transform canvas={yshift=0.2ex}]{\scriptstyle \bot} \arrow[transform canvas={yshift=-0.7ex}, hookleftarrow]{r}[swap]{\forget_\S}
& \S .
\end{tikzcd} \]
\end{prop}

\begin{proof}
Observe that a model structure on an $\infty$-category is clearly determined by its subcategories of weak equivalences and of cofibrations (just as a model structure on a 1-category).  So, it only remains to check that all maps in $(\Cati)_\Thomason$ are cofibrations.  For this, it suffices to present any map in $\CSS$ as the image under $s\S \xra{\leftloc_\CSS} \CSS$ of a cofibration in $s\S_\KQ$.  Note that the inclusion $s\Set \subset s\S$ induces an inclusion $\bC^{s\Set}_\KQ \subset \bC^{s\S}_\KQ$; in fact, we will present any map in $\CSS$ as the image of a cofibration in $s\Set_\KQ$.

To accomplish this, we begin by observing that the composite functor
\[ s\Set \hookra s\S \xra{\leftloc_\CSS} \CSS \]
consists of a right adjoint followed by a left adjoint.  In fact, these are both presented by Quillen functors: the first functor is presented by the right Quillen functor in the evident Quillen adjunction
\[ \pi_0^\lw : s(s\Set_\KQ)_\Reedy \adjarr s\Set_\triv : \const^\lw \]
(where in the left adjoint we slightly abuse notation by using the symbol $\pi_0$ to refer to the composite $s\Set \ra \loc{s\Set}{\bW_\KQ} \simeq \S \xra{\pi_0} \Set$), while the second functor is presented by the left Quillen functor in the left Bousfield localization
\[ \id_{ss\Set} : s(s\Set_\KQ)_\Reedy \adjarr ss\Set_\Rezk : \id_{ss\Set} \]
(see (the proof of) \cite[Theorem 7.2]{RezkCSS}).  Moreover, all objects of $s\Set_\triv$ are fibrant while all objects of $s(s\Set_\KQ)_\Reedy$ are cofibrant, so this composite of Quillen functors does not need to be corrected at either stage in order to compute the value of the corresponding adjoint functor of $\infty$-categories.  On the other hand, this composite functor
\[ s\Set_\triv \xra{\const^\lw} s(s\Set_\KQ)_\Reedy \xra{\id_{ss\Set}} ss\Set_\Rezk \]
of model categories is precisely the functor underlying the left Quillen equivalence $\const^\lw : s\Set_\Joyal \ra ss\Set_\Rezk$ of \cite[Theorem 4.11]{JT}.  It therefore follows that the above composite functor of $\infty$-categories induces an equivalence
\[ \loc{s\Set}{\bW_\Joyal} \xra{\sim} \CSS \]
from the localization of $s\Set$ at the subcategory $\bW_\Joyal \subset s\Set$ to the $\infty$-category of complete Segal spaces.\footnote{Of course, we already knew that $\loc{s\Set}{\bW_\Joyal}$ was equivalent to $\CSS$ (since it is equivalent to $\Cati$); the new information here is that this particular composite functor (of $\infty$-categories) \textit{also} induces this equivalence.}

We can now easily achieve our goal.  Since $\bC^{s\Set}_\KQ = \bC^{s\Set}_\Joyal$, it follows that we can simply choose a cofibration in $s\Set_\Joyal$ presenting our given map in $\CSS \simeq \Cati$: considering this map of simplicial sets as a map of discrete simplicial spaces, its image under the functor $s\S \xra{\leftloc_\CSS} \CSS$ is precisely the chosen map.
\end{proof}

\begin{rem}\label{different Thomasons}
We observe that the Thomason model structure on $\Cati$ does \textit{not} extend the original Thomason model structure on $\strcat$: the model category $\strcat_\Thomason$ is not a model subcategory of the model $\infty$-category $(\Cati)_\Thomason$ (in the sense of \cref{sspaces:defn model subcat}, and ignoring the fact that $\strcat$ is not, strictly speaking, a subcategory of $\Cati$ at all).  However, the weak equivalences remain unchanged: the subcategory $\bW^\strcatsup_\Thomason \subset \strcat$ is pulled back from the subcategory $\bW^\Cati_\Thomason \subset \Cati$ along the composite functor
$ \strcat \ra \Cat \hookra \Cati$.
To illustrate this, we recall the history of this classical model category.

To begin, we recall the main point: categories can individually be considered as ``presentations of spaces'' (via simplicial sets) via the nerve functor $\Nerve : \strcat \ra s\Set$.  Thus, it is natural to wonder whether this can be extended to a global presentation of the \textit{homotopy category} of spaces (i.e.\! $\ho(\S) \simeq s\Set [ \bW_\KQ^{-1}]$) as some localization of the category $\strcat$ of categories.  As a first step in this direction, it was proved in \cite[3.3.1]{Illusie-cotangent} (but attributed there to Quillen) that the nerve functor $\strcat \xra{\Nerve} s\Set$ does indeed induce an equivalence $\strcat[\bW_\Thomason^{-1}] \xra{\sim} s\Set[\bW_\KQ^{-1}]$ on (1-categorical) localizations.  In other words, the relative category $(\strcat,\bW_\Thomason) \in \strrelcat$ has as its (1-categorical) localization the homotopy category $\ho(\S)$ of spaces, as desired.

On the other hand, relative categories are not so easy to work with, and so one might then further wonder whether this relative category structure can be promoted to a model category structure.  Now, the most obvious way that one might hope to obtain this would be to simply lift the classical Kan--Quillen model structure (as in \cref{sspaces:lift cofgen}, but of course really just using \cite[Theorem 11.3.2]{Hirsch}) along the adjunction $\leftloc_\strcatsup : s\Set \adjarr \strcat : \Nerve$.  However, it is easy to see that such a Quillen adjunction could not possibly be a Quillen equivalence: the fibrant objects of $\strcat$ would be precisely the subcategory $\strgpd \subset \strcat$ of groupoids, but these (or rather their nerves) do not model all objects of $s\Set[\bW_\KQ^{-1}]$, and so the derived right adjoint could not possibly be surjective.

But in \cite{ThomasonCat}, Thomason showed that if we instead lift the Kan--Quillen model structure along the composite right adjoint
\[ s\Set_\KQ \xla{\Ex} s\Set \xla{\Ex} s\Set \xla{\Nerve} \strcat , \]
then the resulting Quillen adjunction \textit{is} a Quillen equivalence.\footnote{Heuristically, one might say that the $\Ex$ functor ``makes more simplicial sets fibrant''.  Indeed, recall that it comes equipped with a natural transformation $\id_{s\Set} \ra \Ex$, and the resulting transfinite composition defines a fibrant replacement functor on $s\Set_\KQ$ (see \cite[Chapter III, \sec 4]{GJ}).}\footnote{The $\Ex$ functor is not only a right Quillen equivalence from $s\Set_\KQ$ to itself, but it is also a relative functor -- indeed, it defines a weak equivalence in $\strrelcat_\BarKan$.  Thus, even though here we are (crucially!) not only applying it to fibrant objects of $s\Set_\KQ$, in the end this composite still presents an equivalent map in $\Cati$ (namely $\S \xra{\id_\S} \S$).}  This defines what is now called the \textit{Thomason model structure} on $\strcat$.  From this description, it is clear that the model structure $(\Cati)_\Thomason$ does not extend the model structure $\strcat_\Thomason$.  For instance, it follows from \cref{Thomason is left localizn} that all objects of $(\Cati)_\Thomason$ are cofibrant, whereas \cite[Proposition 5.7]{ThomasonCat} asserts that all cofibrant objects of $\strcat_\Thomason$ are in fact posets.  Moreover, that same result also implies that their notions of fibrancy disagree: it follows from it that any bifibrant object of $\strcat_\Thomason$ is a poset, whereas according to \cref{Thomason fibcy} the fibrant objects of $(\Cati)_\Thomason$ are precisely the $\infty$-groupoids.\footnote{As posets are gaunt, the composites $\strcat \ra \Cati \xra{\Nervei} s\S$ and $\strcat \xra{\Nerve} s\Set \hookra s\S$ are equivalent on such objects.}
\end{rem}

\begin{rem}
One way to interpret \cref{different Thomasons} is to say that the model $\infty$-category $(\Cati)_\Thomason$ entirely accounts for the quirky definition of the model category $\strcat_\Thomason$: in this sense, the only ``obstruction'' to obtaining a model structure on $\strcat$ presenting $\S$ by lifting directly along the nerve functor is the lack of would-be fibrant objects in $\strcat$.  While not every space is presented by a groupoid, certainly every space is presented by an $\infty$-groupoid (!), and so this obstruction vanishes when we pass from $\strcat$ to $\Cati$.
\end{rem}

\begin{rem}\label{Thomason fibns}
In contrast with \cref{Thomason fibcy}, it is not so straightforward to characterize which maps in $(\Cati)_\Thomason$ are fibrations.\footnote{Nor should this necessarily be expected to be straightforward: by the lifting axiom {\liftingaxiom} for model $\infty$-categories and \cref{Thomason is left localizn}, we have $\bF_\Thomason = \rlp( (\bW \cap \bC)_\Thomason) = \rlp(\bW_\Thomason)$.}  By definition, this would be a functor $\C \xra{F} \D$ in $\Cati$ such that the corresponding map
\[ \Nervei(\C) \xra{\Nervei(F)} \Nervei(\D) \]
on nerves in $\CSS \subset s\S$ has $\rlp(I_\KQ)$.  On the one hand, arguments similar to those of \cref{Thomason fibcy} imply that these maps likewise admit \textit{unique} lifts for the \textit{inner} horn inclusions.  On the other hand, the condition that this map have the right lifting property against the outer horn inclusions seems to be a good deal more subtle.
\begin{itemize}

\item At $n=1$, the requirement that our map have the right lifting property against the outer horn inclusion $\Lambda^1_0 \ra \Delta^1$ (resp.\! the outer horn inclusion $\Lambda^1_1 \ra \Delta^1$) is equivalent to the condition that for all objects $c \in \C$, the functor $\C_{/c} \ra \D_{/F(c)}$ (resp.\! the functor $\C_{/c} \ra \D_{/F(c)}$) is surjective.

 \item At $n=2$, even just the requirement that our map have the right lifting property against the outer horn inclusion $\Lambda^2_0 \ra \Delta^2$ already implies that the equivalences in $\C$ are created by the functor $\C \xra{F} \D$.  However, this does not appear to be a sufficient condition.  Of course, one can at least rephrase the condition as follows: for our map to have the right lifting property against $\Lambda^2_0 \ra \Delta^2$ (resp.\! $\Lambda^2_2 \ra \Delta^2$), it must be the case that, given any two maps $\varphi$ and $\psi$ in $\C$ whose sources (resp.\! targets) have been identified, then any factorization of one of the maps $F(\varphi)$ or $F(\psi)$ in $\D$ through the other must already exist in $\C$.

\item Requiring that our map have the right lifting property against the higher outer horn inclusions appears to demand similar but even more exotic properties of our original functor $\C \xra{F} \D$.
 \end{itemize}
 \end{rem}

\begin{rem}\label{Thomason acyclic fibns are all equivces}
Combining \cref{Thomason is left localizn} with the discussion of \cref{sspaces:left localizations give model structures}, we obtain that $(\bW \cap \bF)_\Thomason = (\Cati)^\simeq \subset \Cati$; in other words, any fibration in $(\Cati)_\Thomason$ which induces an equivalence on groupoid completions must in fact itself be an equivalence.  In light of the discussion of \cref{Thomason fibns} regarding the complexities of the subcategory $\bF_\Thomason \subset \Cati$, this appears to be a rather nontrivial fact.\footnote{It is not so hard to see directly that if $\D \in \S \subset \Cati$, then a functor $\C \xra{F} \D$ is in $\bF_\Thomason$ if and only if we also have $\C \in \S \subset \Cati$.  (Thus, $\S_\triv \subset (\Cati)_\Thomason$ is a model subcategory.)  From here, it is clear at least that $\left( (\bW \cap \bF)_\Thomason \cap \S \right) = \S^\simeq \subset \Cati$.}
\end{rem}

\begin{rem}
Both Thomason model structures $(\Cati)_\Thomason$ and $\strcat_\Thomason$ are rather quirky in their own ways.  On the other hand, recalling \cref{pullback of cocart fibn}, it appears that there should exist the structure of an ``$\infty$-category of fibrant objects'' structure on $\Cati$ (or a ``category of fibrant objects'' structure on $\strcat$), in which the co/cartesian fibrations $\D \ra \C$ classified by functors $\C \ra \Cati$ that have property {\propQ} are among the fibrations.\footnote{With the model $\infty$-category $(\Cati)_\Thomason$ in hand, this should be obtainable from arguments along the lines of those contained in \cite{BrownKSthesis}.}  In some vague sense, this would appear to be a more ``true'' articulation of the role of $\Cati$ (or of $\strcat$) as a presentation of $\S$ than either of the corresponding Thomason model structures.
\end{rem}

\bibliographystyle{amsalpha}
\bibliography{gr}{}

\end{document}